\documentclass[a4paper,10pt,reqno]{amsart}
\usepackage{amsmath}
\usepackage{cases}
\usepackage{amsfonts}
\usepackage[colorlinks,linkcolor=blue,citecolor=blue]{hyperref}
\usepackage{latexsym, amssymb, amsmath, amsthm, bbm}
\usepackage[all]{xy}
\usepackage{pgfplots}
\usepackage{enumerate}
\usepackage{mathrsfs}

\DeclareSymbolFont{EulerExtension}{U}{euex}{m}{n}
\DeclareMathSymbol{\euintop}{\mathop} {EulerExtension}{"52}
\DeclareMathSymbol{\euointop}{\mathop} {EulerExtension}{"48}

\allowdisplaybreaks[4]

\def \id{\operatorname{id}}

\def \C{\mathcal{C}}

\def \Z{\mathbb{Z}}

\def \k{\Bbbk}

\def \dim{\operatorname{dim}}

\def \Ext{\operatorname{Ext}}

\def \C{\mathcal{C}}
\def \D{\Delta}

\def \span{\operatorname{span}}
\def \A{\mathcal{A}}
\def \B{\mathcal{B}}
\def \C{\mathcal{C}}
\def \D{\mathcal{D}}
\def \E{\mathcal{E}}
\def \F{\mathcal{F}}
\def \G{\mathcal{G}}
\def \X{\mathcal{X}}
\def \Y{\mathcal{Y}}
\def \Z{\mathcal{Z}}
\def \W{\mathcal{W}}

\usepackage{mathtools}

\numberwithin{equation}{section}

\newtheorem{theorem}{Theorem}[section]
\newtheorem{lemma}[theorem]{Lemma}
\newtheorem{proposition}[theorem]{Proposition}
\newtheorem{corollary}[theorem]{Corollary}
\newtheorem{definition}[theorem]{Definition}
\newtheorem{example}[theorem]{Example}
\newtheorem{remark}[theorem]{Remark}

\newtheorem{conjecture}[theorem]{Conjecture}
\newtheorem{notation}[theorem]{Notation}

\begin{document}
\title[Hopf algebras of finite corepresentation type]{Hopf algebras with the dual Chevalley property of finite corepresentation type}
\author[J. Yu]{Jing Yu}
\author[K. Li]{Kangqiao Li$^\dag$}
\author[G. Liu]{Gongxiang Liu}
\address{Department of Mathematics, Nanjing University, Nanjing 210093, China}
\email{dg21210018@smail.nju.edu.cn}
\address{School of Mathematics, Hangzhou Normal University, Hangzhou 311121, China}
\email{kqli@hznu.edu.cn}
\address{Department of Mathematics, Nanjing University, Nanjing 210093, China}
\email{gxliu@nju.edu.cn}

\thanks{2020 \textit{Mathematics Subject Classification}. 16T05, 16G60 (primary), 16G20, 16G10 (secondary).}
\thanks{$^\dag$ Corresponding author}
\keywords{Hopf algebras, Dual Chevalley property, Link quiver, Finite corepresentation type}
\maketitle

\date{}

\begin{abstract}
Let $H$ be a finite-dimensional Hopf algebra over an algebraically closed field $\k$ with the dual Chevalley property. We prove that $H$ is of finite corepresentation type if and only if it is coNakayama, if and only if the link quiver $\mathrm{Q}(H)$ of $H$ is a disjoint union of basic cycles, if and only if the link-indecomposable component $H_{(1)}$ containing $\k1$ is a pointed Hopf algebra and the link quiver of $H_{(1)}$ is a basic cycle.
\end{abstract}
\maketitle
\section{Introduction}
In the view point of representation type, every finite-dimensional algebra exactly belongs to one of following three kinds of algebras: algebras of finite representation type, algebras of tame types and wild algebras (See \cite{Dro86}). From then on, the classification for a given kind of algebras according to their representation type has received considerable attention. See, for example, \cite{Ari05, Ari17, Ari21, DEMN99, EN01, KOS11, Rin75, Rin78}.
In the case of Hopf algebras, much effort was put in pointed Hopf algebras or their dual, that is, elementary Hopf algebras. In the case of modular group algebras of finite groups, the authors in \cite{Ben98, BD82, Erd90, Hig54} show that a block of such modular group algebra is of finite representation type if and only if the corresponding defect groups are cyclic and while it is tame if and only if $\operatorname{char}\k=2$ and its defects groups are dihedral, semidihedral and generalized quaternion. The classification of small quantum groups according to their representation type can be found in \cite{Cil97, Sut94, Xia97}. They show that the only tame one is $u_q(\mathfrak{sl}_2)$ and others are all wild. For cocommutative Hopf algebras, Farnsteiner and his collaborators have classify all finite-dimensional cocommutative Hopf algebras, i.e., finite algebraic groups, of finite representation type and tame type \cite{Far06, FS02, FS07, FV00, FV03}. The third author and his collaborators get the classification of elementary Hopf algebras according to their representation type from 2006 to 2013 \cite{LL07,HL09, Liu06, Liu13}. We note that there is indeed a common point in above classification: a finite-dimensional (cocommutative, elementary) Hopf algebra is of finite representation type if and only if it is a Nakayama algebra. We cannot help but hope that this observation holds true for more Hopf algebras.

Among of these results, constructing Hopf algebras structures through using quivers was shown to be a very effective way, which is due to the works of Cibils-Rosso for pointed case and  Green-Solberg for elementary case \cite{Cil93, CR97, CR02, GS98}.
As a development, in \cite{CHYZ04, OZ04}, the authors give a classification of non-semisimple monomial Hopf algebras and get more.
In 2007, the third author and Li \cite{LL07} have classified all finite-dimensional pointed Hopf algebras of finite corepresentation type and show that they are all monomial Hopf algebras \cite[Theorem 4.6]{LL07}.
At the same time, it is well known that in the representation of finite-dimensional algebras, the Ext quiver is a fundamental tool. The Ext quiver of a coalgebra has been introduced by Chin and Montgomery in \cite{CM97} too. Montgomery also introduced the link quiver of coalgebra $H$ by using the wedge of simple subcoalgebras of $H$ (see \cite[Definition 1.1]{Mon95}). In \cite[Definition 4.1]{CHZ06}, the definition of link quiver has been modified. In addition, the authors of \cite{CHZ06} unified the link quiver of a coalgebra with the Ext quiver. Obviously, these quivers are not limited to elementary or pointed Hopf algebras.

 We know that the Hopf algebras with the (dual) Chevalley property is a kind of natural generalization of elementary (pointed) Hopf algebras.
 These Hopf algebras have been studied by many authors. See, for example, \cite{ABM12, AEG01, Li22a, Li22b, LL22, LZ19}. In \cite{AGM17, CDMM04, Mom13, ZGH21}, the authors present some explicit examples of Hopf algebras with the dual Chevalley property.

 The aim of this paper and subsequent works is to classify finite-dimensional Hopf algebras with the dual Chevalley property according to their corepresentation type. Here by the dual Chevalley property we mean that the coradical $H_0$ is a Hopf subalgebra. The main tool we want to use is the link quiver. One of key points of this paper is that one can describe the structure of the link quiver by applying multiplicative matrices and primitive matrices now, which are developed by the second author and his collaborator \cite{LZ19,LL22,Li22a,Li22b}. Besides, we attempt to generalize above stated result \cite[Theorem 4.6]{LL07} in order to give the structure of finite-dimensional Hopf algebras with the dual Chevalley property of finite corepresentation type. We also conjecture (see Conjecture \ref{4.11}) that the link quiver of $H$ should share similar symmetry with a Hopf quiver (\cite{CR02}), which will guide our future research as well. For example, in the future, we can consider the question that on which quivers one can construct graded Hopf algebras with the dual Chevalley property.

Our main results are Theorems \ref{coro:equ}, \ref{thm:finitecorep}, \ref{thm:finitecorepp} and Corollary \ref{coro:conakayama}, stating that:
\begin{theorem}
Let $H$ be a finite-dimensional non-cosemisimple Hopf algebra over an algebraically closed field $\k$ with the dual Chevalley property and $\mathrm{Q}(H)$ be the link quiver of $H$. Then the following statements are equivalent:
\begin{itemize}
  \item[(1)]$H$ is of finite corepresentation type;
  \item[(2)]Every vertex in $\mathrm{Q}(H)$ is both the start vertex of only one arrow and the end vertex of only one arrow, that is, $\mathrm{Q}(H)$ is a disjoint union of basic cycles;
  \item[(3)]There is only one arrow $C\rightarrow \k1$ in $\mathrm{Q}(H)$ whose end vertex is $\k1$ and $\dim_{\k}(C)=1$;
  \item[(4)]There is only one arrow $\k1\rightarrow D$ in $\mathrm{Q}(H)$ whose start vertex is $\k1$ and $\dim_{\k}(D)=1$.
\end{itemize}
\end{theorem}

\begin{corollary}
A finite-dimensional Hopf algebra $H$ over an algebraically closed field $\k$ with the dual Chevalley property is of finite corepresentation type if and only if $H$ is coNakayama.
\end{corollary}

\begin{theorem}
Let $\k$ be an algebraically closed field of characteristic 0. Then a finite-dimensional Hopf algebra $H$ over $\k$ with the dual Chevalley property is of finite corepresentation type if and only if either of the following conditions is satisfied:
\begin{itemize}
  \item[(1)]$H$ is cosemisimple;
  \item[(2)]$H$ is not cosemisimple and $H_{(1)}\cong A(n, d, \mu, q)$.
  \end{itemize}
\end{theorem}

\begin{theorem}
Let $\k$ be an algebraically closed field of positive characteristic $p$. Then a finite-dimensional Hopf algebra $H$ over $\k$ with the dual Chevalley property is of finite corepresentation type if and only if either of the following conditions is satisfied:
\begin{itemize}
  \item[(1)]$H$ is cosemisimple;
  \item[(2)]$H$ is not cosemisimple and $H_{(1)}\cong C_d(n)$.
  \end{itemize}
\end{theorem}
In \cite[Question 4.13]{Li22a}, the author pose the question of whether the Hopf subalgebra $H_{(1)}$ is normal in $H$ or not.
As a byproduct, we give a negative answer to this question. This also tells us that the Hopf algebras with dual Chevalley property is really a nontrivial generalization of pointed Hopf algebras in general.

The organization of this paper is as follows: In Section \ref{section2}, we recall the definition of multiplicative and primitive matrices. The properties of a complete family of non-trivial $(\C, \D)$-primitive matrices are provided. In Section \ref{section3}, we construct a complete family of non-trivial primitive matrices in two ways. Note that the cardinal number of a complete family of non-trivial $(\C, \D)$-primitive matrices coincides with the number of arrows from vertex $D$ to vertex $C$ in the link quiver. Then we discuss the properties for the link quiver of a Hopf algebra with the dual Chevalley property in Section \ref{section4}. We devote Section \ref{section5} to give proofs of our main results: Theorems \ref{coro:equ}, \ref{thm:nakayama}, \ref{thm:finitecorep} and \ref{thm:finitecorepp}. At last, some examples and applications are given in Section \ref{section6}. In particular, Example \ref{example6.1} gives a negative answer to \cite[Question 4.13]{Li22a}.

\section{Matrices and bicomodules over coalgebra}\label{section2}
\subsection{Preliminaries}
Throughout this paper $\k$ denotes an \textit{algebraically closed field} and all spaces are over $\k$. The tensor product over $\k$ is denoted simply by $\otimes$. We refer to \cite{Mon93} for the basics about Hopf algebras.

The concept of multiplicative matrices was introduced by Manin in \cite{Man88}. Later in 2019, Li and Zhu \cite{LZ19} introduced the concept of primitive matrices. Recently, more properties of multiplicative matrices and primitive matrices have been observed. The authors of \cite{Li22a, Li22b, LL22, LZ19} used these two notions to generalize some results of pointed Hopf algebras to non-pointed ones.

Let us first recall the definition of multiplicative matrices.

\begin{definition}\emph{(}\cite[Definition 2.3]{Li22a}\emph{)}
Let $(H,\Delta,\varepsilon)$ be a coalgebra over $\k$.
\begin{itemize}
  \item[(1)] A square matrix $\G=(g_{ij})_{r\times r}$ over $H$ is said to be multiplicative, if for any $1\leq i,j \leq r$, we have $\Delta(g_{ij})=\sum\limits_{t=1}^r g_{it}\otimes g_{tj}$ and $\varepsilon(g_{ij})=\delta_{i, j}$, where $\delta_{i, j}$ denotes the Kronecker notation;
  \item[(2)] A multiplicative matrix $\C$ is said to be basic, if its entries are linearly independent.
\end{itemize}
\end{definition}
Multiplicative matrices over a coalgebra can be understood as a generalization of group-like elements. We know that all the entries of a basic multiplicative matrix $\C$ span a simple subcoalgebra $C$ of $H$. Conversely, for any simple coalgebra $C$ over
$\k$, there exists a basic multiplicative matrix $\C$ whose entries span $C$ (for details, see \cite{LZ19}, \cite{Li22a}).
And according to \cite[Lemma 2.4]{Li22a}, the basic multiplicative matrix of the simple coalgebra $C$ would be unique up to the similarity relation. More specifically, suppose that $\C$ is a basic multiplicative matrix of the simple coalgebra $C$. Then $\C^\prime$ is also a basic multiplicative matrix of $C$ if and only if there exists an invertible matrix $L$ over $\k$ such that $\C^\prime=L\C L^{-1}$.

Next we recall the definition of primitive matrices, which is a non-pointed analogue of primitive elements.
\begin{definition}\emph{(}\cite[Definition 3.2]{LZ19} and \cite[Definition 4.4]{Li22b}\emph{)}
Let $(H,\Delta,\varepsilon)$ be a coalgebra over $\k$. Suppose $\C=(c_{ij})_{r\times r}$ and $\D=(d_{ij})_{s\times s}$ are basic multiplicative matrices over $H$.
\begin{itemize}
  \item[(1)] A matrix $\X=(x_{ij})_{r\times s}$ over $H$ is said to be $(\C, \D)$-primitive, if $$\Delta(x_{ij})=\sum\limits_{k=1}^r c_{ik}\otimes x_{kj}+\sum\limits_{t=1}^s x_{it}\otimes d_{tj}$$ holds for any $1\leq i,j \leq r$;
  \item[(2)] A primitive matrix $\X$ is said to be non-trivial, if there exists some entry of $\X$ which does not belong to the coradical $H_0$.
\end{itemize}
\end{definition}
Let $C, D$ be the simple subcoalgebras spanned by the entries of basic multiplicative matrices $\C$ and $\D$, respectively. For any $(\C, \D)$-primitive matrix $\X$, it is evident that all the entries of $\X$ must belong to $C\wedge D$ and automatically belong to $H_1:=H_0 \wedge H_0$, where $H_0$ is the coradical of $H$.
\subsection{Non-trivial primitive matrices and simple bicomodules over a coalgebra}
In this subsection, let $(H,\Delta,\varepsilon)$ be a coalgebra over $\k$. Denote the coradical filtration of $H$ by $\{H_n\}_{n\geq0}$ and the set of all the simple subcoalgebras of $H$ by $\mathcal{S}$. For any simple subcoalgebra $C\in \mathcal{S}$, we fix a basic multiplicative matrix $\C$ of $C$.

For any matrix $\X=\left(x_{ij}\right)_{r\times s}$ over $H$,  denote the matrix $\left(\overline{x_{ij}}\right)_{r\times s}$ by $\overline{\X}$, where $\overline{x_{ij}}=x_{ij}+H_0\in H/H_0$. Besides, the subspace of $H/H_0$ spanned by the entries of $\overline{\X}$ is denoted by $\operatorname{span}(\overline{\X})$.

We start this subsection by giving the following lemma, which describes a property of simple bicomodules.
\begin{lemma}\label{lemma:simple bicomodule}
For any $C,D\in\mathcal{S}$ with $\dim_{\k}(C)=r^2, \dim_{\k}(D)=s^2$, if $M$ is a simple $C$-$D$-bicomodule, then $\dim_{\k}(M)=rs$.
\end{lemma}
\begin{proof}
Since $C^*$ and $D^*$ are central simple algebras, it follows that $D^*\otimes C^{*\mathrm{op}}$ is also a central simple algebra and  $$D^*\otimes C^{*\mathrm{op}}\cong M_{rs}(\k)$$
 as algebras, where $M_{rs}(\k)$ is a matrix algebra. It is known that the dimension of simple left $M_{rs}(\k)$-modules is $rs$. Besides, the category of finite-dimensional left $D^*\otimes C^{*\mathrm{op}}$-modules, the category of finite dimensional $D^*$-$C^*$-bimodules and the category of finite-dimensional $C$-$D$-bicomodules are isomorphic. And the isomorphisms preserve the dimension. Hence the dimension of the simple $C$-$D$-bicomodule $M$ is $rs$.
\end{proof}
Let $$\pi: H_1 \longrightarrow H_1/H_0$$ be the quotient map. For any $\overline{h}\in H_1/H_0$, define
\begin{eqnarray}\label{comodulestructure}
\rho_L(\bar{h})=(\id\otimes\pi)\Delta(h),\;\;\rho_R(\bar{h})=(\pi\otimes\id)\Delta(h).
\end{eqnarray}
It is evident that $(H_1/H_0, \rho_L, \rho_R)$ is an $H_0$-bicomodule. Now we turn to mention  $\span(\overline{\X})$, where $\X$ is a non-trivial $(\C, \D)$-primitive matrix.
\begin{lemma}\label{lemma:rs}
For any $C,D\in\mathcal{S}$ with $\dim_{\k}(C)=r^2, \dim_{\k}(D)=s^2$, if $\X_{r\times s}=\left(x_{ij}\right)_{r\times s}$ is a non-trivial $(\C, \D)$-primitive matrix, then $\span(\overline{\X})$ is a simple $C$-$D$-bicomodule. Moreover, $\dim_{\k}(\span(\overline{\X}))=rs$.
\end{lemma}

\begin{proof}
By \cite[Proposition 2.11]{Li22a}, we know that $x_{ij}\notin H_0$ holds for all $1\leq i\leq r$ and $1\leq j\leq s$. Notice that
$$\rho_L(\overline{x_{ij}})=(\id\otimes \pi)\Delta(x_{ij})=\sum\limits_{k=1}^r c_{ik}\otimes \overline{x_{kj}} ,$$
$$\rho_R(\overline{x_{ij}})=(\pi\otimes \id)\Delta(x_{ij})=\sum\limits_{t=1}^s \overline{x_{it}}\otimes d_{tj} .$$
It is straightforward to show that $(\span(\overline{\X}), \rho_L, \rho_R)$ is a $C$-$D$-bicomodule and $$\dim_{\k}(\span(\overline{\X}))\leq rs.$$ But according to Lemma \ref{lemma:simple bicomodule}, the dimension of any $C$-$D$-sub-bicomodule is at least $rs$. Thus we conclude that $$\dim_{\k}(\span(\overline{\X}))=rs$$ and $\span(\overline{\X})$ is a simple $C$-$D$-bicomodule.
\end{proof}
A direct consequence of this lemma is:
\begin{corollary}\label{coro:different X}
If $\X$ and $\X^\prime$ are non-trivial $(\C,\D)$-primitive matrices over $H$, then either $\span(\overline{\X})\cap\span(\overline{\X^\prime})=0$ or $\span(\overline{\X})=\span(\overline{\X^\prime})$.
\end{corollary}
\begin{proof}
According to Lemma \ref{lemma:rs}, it follows that $\span(\overline{\X})$ and $\span(\overline{\X^\prime})$ are both $C$-$D$-bicomodules. It is clear that $\span(\overline{\X})\cap\span(\overline{\X^\prime})$ is a sub-$C$-$D$-bicomodule of $\span(\overline{\X})$. But since $\span(\overline{\X})$ is simple, its  sub-$C$-$D$-bicomodule $\span(\overline{\X})\cap\span(\overline{\X^\prime})$ is either $\span(\overline{\X})$ or $0$. In the previous case, $\span(\overline{\X})\supseteq\span(\overline{\X^\prime})$. By the same taken, we can prove that $\span(\overline{\X})\subseteq\span(\overline{\X^\prime})$.
\end{proof}

Moreover, there are further properties for non-trivial primitive matrices.
\begin{corollary}\label{coro:linear independent}
Let $C, D\in\mathcal{S}$ with basic multiplicative matrices $\C_{r\times r}$ and $\D_{s\times s}$, respectively. Suppose $\X:=(x_{ij})_{r\times s}$ is a $(\C, \D)$-primitive matrix. Then the followings are equivalent:
\begin{itemize}
  \item[(1)] $\X$ is non-trivial; 
  \item[(2)] $x_{ij}\notin H_0$ holds for all $1\leq i\leq r$ and $1\leq j\leq s$;
  \item[(3)] $\{x_{ij}\mid 1\leq j\leq s\}$ are linearly independent in $H_1/H_0$ (the quotient space) for each $1\leq i\leq r$, and $\{x_{ij}\mid 1\leq i\leq r\}$ are linearly independent in $H_1/H_0$ for each $1\leq j\leq s$.
  \item[(4)] $\{x_{ij} \mid 1\leq j\leq s,1\leq i\leq r\}$ are linearly independent in $H_1/H_0$.
\end{itemize}
\end{corollary}
\begin{proof}
The equivalence of $(1), (2)$ and $(3)$ is by \cite[Proposition 2.11]{Li22a}. And $(4)$ clearly implies $(1), (2)$ and $(3)$. To complete the proof, we only need to show that $(1)$ implies $(4)$. Note that if $\X$ is non-trivial, it follows from Lemma \ref{lemma:rs} that $\span(\overline{\X})$ is a simple $C$-$D$-bicomodule and $\dim_{\k}(\span(\overline{\X}))=rs$, which means that $\{x_{ij} \mid 1\leq j\leq s,1\leq i\leq r\}$ are linearly independent in $H_1/H_0$.
\end{proof}

Recall that $\{e_C\}_{C\in \mathcal{S}}$ is called a family of \textit{coradical orthonormal idempotents} (see \cite[Section 1]{Rad78}) in $H^*$, if $$e_C|_D=\delta_{C,D}\varepsilon|_D,\;\;\;\;e_Ce_D=\delta_{C,D}e_C\;\;\;\;
      (\text{for any}\;C,D\in\mathcal{S}),\;\;\;\;\sum\limits_{C\in\mathcal{S}}e_C=\varepsilon.$$
The existence of a family of coradical orthonormal idempotents is affirmed in \cite[Lemma 2]{Rad78}.
About more properties of coradical orthonormal idempotents, the reader is referred to \cite[Proposition 2.2]{LZ19} for details.
We use the notations below for convenience:
      $${}^Ch=h\leftharpoonup e_C,\;\;\;h{}^D=e_D\rightharpoonup h,\;\;\;{}^Ch{}^D=e_D\rightharpoonup h\leftharpoonup e_C\;\;\;(\text{for any}\;h\in H\;\text{and}\;C,D \in \mathcal{S}),$$
where $\rightharpoonup$ and $\leftharpoonup$ are hit actions of $H^*$ on $H$.

Moreover, let $\{e_C\}_{C\in \mathcal{S}}$ be a family of coradical orthonormal idempotents. If $V$ is an $H_0$-$H_0$-bicomodule with left comodule structure $\delta_L$ and right comodule structure $\delta_R$, define
$${}^Cv=v\leftharpoonup e_C=(e_C\otimes \id)\delta_L(v) ,\;\;\;v^D=e_D\rightharpoonup v=(\id\otimes e_D)\delta_R(v),$$
$${}^Cv{}^D=e_D\rightharpoonup v\leftharpoonup e_C\;\;\;(\text{for any}\;v\in V\;\text{and}\;C,D \in \mathcal{S}).
$$

With the notations above, we can establish the following decomposition of $H_1/H_0$ as a direct sum.
\begin{lemma}\label{V=sumCVD}
Suppose that $V$ is an $H_0$-$H_0$-bicomodule, then $V=\bigoplus\limits_{C, D\in\mathcal{S}}{{}^C V{}^D}$,
where ${{}^C V{}^D}=e_D\rightharpoonup V\leftharpoonup e_C$ is a $C$-$D$-bicomodule. In particular, we have $H_1/H_0=\bigoplus\limits_{C, D\in\mathcal{S}}{}^C ({H_1/H_0}){}^D$.
\end{lemma}

\begin{proof}
It is straightforward to show that ${{}^C V{}^D}$ is a $C$-$D$-bicomodule. For any $v\in V$, since $\sum\limits_{C\in\mathcal{S}}e_C=\varepsilon$, we have $$v=\varepsilon\rightharpoonup v\leftharpoonup \varepsilon=\sum\limits_{C, D\in \mathcal{S}}{{}^C v{}^D}.$$
Suppose $0=\sum\limits_{C, D\in \mathcal{S}}w_{C, D}$, where $w_{C, D}\in {{}^C V{}^D}$ for any $C, D\in \mathcal{S}$.
Note that for any $E, F\in \mathcal{S}$, we have
\begin{eqnarray*}
0&=&e_E\rightharpoonup 0\leftharpoonup e_F\\
&=&e_E\rightharpoonup (\sum\limits_{C, D\in \mathcal{S}}w_{C, D})\leftharpoonup e_F\\
&=&\sum\limits_{C, D\in \mathcal{S}}e_E\rightharpoonup w_{C, D}\leftharpoonup e_F\\
&=&w_{F, E}.
\end{eqnarray*}
Thus we complete the proof.
\end{proof}

Besides, for any $C, D\in\mathcal{S}$, since $\Delta({}^CH_1{}^D)\subseteq C\otimes {}^CH_1{}^D+{}^CH_1{}^D\otimes D$, it follows that $({}^CH_1{}^D+H_0)/H_0$ is exactly a $C$-$D$-bicomodule with the bicomodule structure $\rho_L, \rho_R$ defined in (\ref{comodulestructure}). Thus we have another direct sum decomposition of $H_1/H_0$ and these two kinds of decomposition are related.
\begin{lemma}\label{lemma:sum2}
As an $H_0$-$H_0$-bicomodule, $H_1/H_0=\bigoplus\limits_{C, D\in\mathcal{S}}({}^CH_1{}^D+H_0)/H_0$. Moreover, ${^C(H_1/H_0)^D}=({}^CH_1{}^D+H_0)/H_0$ holds for any $C, D\in\mathcal{S}$.
\end{lemma}

\begin{proof}
For any $x\in H_1$, a direct computation follows that
\begin{eqnarray*}
e_D\rightharpoonup \overline{x}\leftharpoonup e_C&=&\sum\langle e_C, x_{(1)}\rangle \overline{x_{(2)}}\langle e_D, x_{(3)}\rangle\\
&=&\sum\overline{\langle e_C, x_{(1)}\rangle x_{(2)}\langle e_Dx_{(3)}\rangle}\\
&=&\overline{e_D\rightharpoonup x\leftharpoonup e_C}\\
&\in& ({}^CH_1{}^D+H_0)/H_0,
\end{eqnarray*}
where we use the Sweedler notation $\Delta(x)=\sum x_{(1)}\otimes x_{(2)}$ for the comultiplication.
So we have $${}^C(H_1/H_0){}^D\subseteq ({}^CH_1{}^D+H_0)/H_0$$
and  $$\bigoplus\limits_{C, D\in\mathcal{S}}{}^C ({H_1/H_0}){}^D=H_1/H_0=\sum\limits_{C, D\in\mathcal{S}}({}^CH_1{}^D+H_0)/H_0.$$
The same proof with Lemma \ref{V=sumCVD} can be applied to $H_1/H_0$, then we get $$H_1/H_0=\bigoplus\limits_{C, D\in\mathcal{S}}({}^CH_1{}^D+H_0)/H_0,$$ which implies that $${^C(H_1/H_0)^D}=({}^CH_1{}^D+H_0)/H_0.$$
\end{proof}
For the remaining of this subsection, let $C, D\in\mathcal{S}$ with basic multiplicative matrices $\C=(c_{ij})_{r\times r}$ and $\D=(d_{ij})_{s\times s}$, respectively.
\begin{lemma}\label{Lemma:XinCPD}
For any $(\C, \D)$-primitive matrix $\X$, we have $$\span(\overline{\X})\subseteq({}^CH_1{}^D+H_0)/H_0.$$
\end{lemma}
\begin{proof}
For any $\overline{x}\in\span(\overline{\X})\subseteq H_1/H_0$, it follows from Lemma \ref{lemma:simple bicomodule} that
$$\rho_L(\overline{x})\subseteq C\otimes \span(\overline{\X})$$
and
$$\rho_R(\overline{x})\subseteq \span(\overline{\X})\otimes D.$$
 According to Lemma \ref{lemma:sum2}, we have $$\overline{x}\in \bigoplus\limits_{E, F\in\mathcal{S}} ({}^EH_1{}^F+H_0)/H_0.$$
 Note that $({}^EH_1{}^F+H_0)/H_0$ is a $E$-$F$-bicomodule, for any $E, F\in\mathcal{S}$. It follows that $$\overline{x}\in  ({}^CH_1{}^D+H_0)/H_0,$$
 which means that $$\span(\overline{\X})\subseteq({}^CH_1{}^D+H_0)/H_0.$$
\end{proof}
Next we consider the inverse.
\begin{lemma}\label{Lemmasimplemodule}
 If $W$ is a subspace of ${}^CH_1{}^D+H_0$ such that $\overline{W}$ is a simple $C$-$D$-sub-bicomodule of $({}^CH_1{}^D+H_0)/H_0$, then there exists some non-trivial $(\C, \D)$-primitive matrix $\W$ such that $\span(\overline{\W})=\overline{W}$.
\end{lemma}

\begin{proof}
For any nonzero $\overline{w}\in\overline{W}$, without the loss of generality, we assume $w\in {}^CH_1{}^D$.
\begin{itemize}
  \item[(1)]If $C\neq D$, by \cite[Theorem 3.1(1)]{LZ19} and its proof, we know that there exist $rs$ $(\C, \D)$-primitive matrices $\W^{(i^\prime, j^\prime)}=\left(w_{ij}^{(i^\prime, j^\prime)}\right)_{r\times s}$ $(1\leq i^\prime\leq r, 1\leq j^\prime\leq s) $ such that $$w=\sum\limits_{i=1}^r\sum\limits_{j=1}^s w_{ij}^{(i, j)},$$
      $$\Delta(w)=\sum\limits_{i^\prime,i=1}^r c_{i^\prime i}\otimes x_i^{(i^\prime)}+\sum\limits_{j,j^\prime=1}^s y_{j}^{(j^\prime)}\otimes d_{jj^\prime},$$
      and
      $$\Delta(x_i^{(i^\prime)})=\sum\limits_{k=1}^r c_{ik}\otimes x_k^{(i^\prime)}+\sum\limits_{j,j^\prime=1}^s w_{ij}^{(i^\prime, j^\prime)}\otimes d_{jj^\prime},$$
      where $x_i^{(i^\prime)}, y_j^{(j^\prime)}\in {}^CH_1{}^D\cap \ker \varepsilon$ for all $1\leq i^\prime, i \leq r, 1\leq j^\prime, j \leq s$.
      Observe $\overline{W}$ is a $C$-$D$-sub-bicomodule whose comodule structure is induced by comultiplication, namely,
      $$\rho_L(\overline{w})=(\id\otimes \pi)\Delta(w)\in C\otimes \overline{W}, \;\;\rho_R(\overline{w})=(\pi\otimes \id)\Delta(w)\in \overline{W}\otimes D.$$
      As $\{c_{ii^\prime}\mid 1\leq i, i^\prime \leq r\}$ and $\{d_{jj^\prime}\mid 1\leq j, j^\prime \leq s\}$ are linearly independent, thus $$\overline{x_i^{(i^\prime)}}, \overline{y_{j}^{(j^\prime)}}\in \overline{W}$$ for all $i, i^\prime, j, j^\prime$. According to a similar argument, for any $i^\prime, i$, we have $$\rho_R(\overline{x_i^{(i^\prime)}})\in\overline{W}\otimes D.$$  This means that $\overline{ w_{ij}^{(i^\prime, j^\prime)}}\in\overline{W}$ for all $i, i^\prime, j, j^\prime$. Hence we have  $$\span(\overline{\W^{(i^\prime,j^\prime)}})\subseteq \overline{W}$$ for all $i^\prime, j^\prime$.
      Since $\overline{w}$ is nonzero, there must be some pair $(i_0^\prime, j_0^\prime)$ such that $\W^{(i_0^\prime, j_0^\prime)}$ is a non-trivial $(\C, \D)$-primitive matrix. However, note that $\overline{W}$ is a simple $C$-$D$-sub-bicomodule. It follows that $$\span(\overline{\W^{(i_0^\prime, j_0^\prime)}})=\overline{W}.$$
   \item[(2)]If $C=D$, according to \cite[Theorem 3.1]{LZ19} or \cite[Lemma 2.14(2)]{LL22}, we choose $\C=\D$, and there exist $rs$ $(\C, \D)$-primitive matrices $\W^{(i^\prime, j^\prime)}=\left(w_{ij}^{(i^\prime, j^\prime)}\right)_{r\times s}$ $(1\leq i^\prime\leq r, 1\leq j^\prime\leq s) $ such that $$w-\sum\limits_{i=1}^r\sum\limits_{j=1}^s w_{ij}^{(i, j)}\in C.$$
       Using \cite[Lemma 3.1]{LZ19}, we know that there exists an element $c\in C$ such that
       $$\Delta(w-c)\in C\otimes ({}^CH_1{}^C)^++({}^CH_1{}^C)^+\otimes C,$$
       where $({}^CH_1{}^C)^+={}^CH_1{}^C\cap \ker\varepsilon.$
       Then the same proof of $(1)$ can be applied to the element $w-c$. Thus we can find a non-trivial $(\C, \D)$-primitive matrix $\W^{(i_0^\prime, j_0^\prime)}$ such that $$\span(\overline{\W^{(i_0^\prime, j_0^\prime)}})=\overline{W}.$$
   \end{itemize}
\end{proof}
Clearly, a coalgebra $H$ is cosemisimple if and only if the category of left (resp. right) $H$-comodules is a semisimple category. This means that any $C$-$D$-bicomodule is cosemisimple. Applying Lemma $\ref{lemma:sum2}$ to the cosemisimple $C$-$D$-bicomodule $({}^CH_1{}^D+H_0)/H_0$, we can decompose it into the direct sum of simple $C$-$D$-sub-bicomodules as the following.
\begin{corollary}\label{coro:complete}
There exists a family $\{\X^{(\gamma)}\}_{\gamma\in\mathit{\Gamma}}$ of non-trivial $(\C, \D)$-primitive matrices such that
\begin{eqnarray}\label{equation:direct sum}
{^C(H_1/H_0)^D}=({}^CH_1{}^D+H_0)/H_0=\bigoplus\limits_{\gamma\in\mathit{\Gamma}}\span(\overline{\X^{(\gamma)}}).
\end{eqnarray}
\end{corollary}

\begin{definition}
A family of non-trivial $(\C, \D)$-primitive matrices $\{\X^{(\gamma)}\}_{\gamma\in\mathit{\Gamma}}$ satisfying the property of (\ref{equation:direct sum}) in Corollary \ref{coro:complete} is said to be complete.
\end{definition}
The corollary below is followed immediately by Lemma \ref{lemma:simple bicomodule} and Corollary \ref{coro:complete}.
\begin{corollary}\label{coro:basis}
If $\{\X^{(\gamma)}\}_{\gamma\in\mathit{\Gamma}}$ is a complete family of non-trivial $(\C, \D)$-primitive matrices, where $\X_\gamma=(x_{ij}^{(\gamma)})_{r\times s}$, then $\{\overline{x_{ij}^{(\gamma)}}\mid\gamma\in\mathit{\Gamma}, 1\leq i\leq r, 1\leq j\leq s\}$ is a linear basis of $({}^CH_1{}^D+H_0)/H_0$.
\end{corollary}

A complete family of non-trivial $(\C, \D)$-primitive matrices is the main tool to help us characterize the link quiver of $H$ in the subsequent sections. Thus some of its properties should be noticed.

\begin{lemma}\label{lemma:complement}
Suppose $\{\X^{(\lambda)}\}_{\lambda\in\mathit{\Lambda}}$ is a family of non-trivial $(\C, \D)$-primitive matrices such that the sum
$\sum\limits_{\lambda\in\Lambda}\span(\overline{\X^{(\lambda)}})$ in $({}^CH_1{}^D+H_0)/H_0$ is direct. Then we can find a complete family of non-trivial $(\C, \D)$-primitive matrices
$\{\X^{(\gamma)}\}_{\gamma\in\mathit{\Gamma}}$ such that $\{\X^{(\lambda)}\}_{\lambda\in\mathit{\Lambda}}$ is a subset of $\{\X^{(\gamma)}\}_{\gamma\in\mathit{\Gamma}}$.
 \end{lemma}

\begin{proof}
Let $M$ be a complement of $\bigoplus_{\lambda\in\Lambda}\span(\overline{\X^{(\lambda)}})$ in $({}^CH_1{}^D+H_0)/H_0$. According to Lemma \ref{Lemmasimplemodule}, we can show that $$M=\bigoplus_{\gamma^\prime\in\Gamma^\prime}\span(\overline{\X^{(\gamma^\prime)}})$$ for some non-trivial $(\C, \D)$-primitive matrices $\{\X^{(\gamma^\prime)}\}_{\gamma^\prime\in\Gamma^\prime}$. Let $$\{\X^{(\gamma)}\}_{\gamma\in\mathit{\Gamma}}=\{\X^{(\lambda)}\}_{\lambda\in\mathit{\Lambda}}\cup\{\X^{(\gamma^\prime)}\}_{\gamma^\prime\in\Gamma^\prime}.$$
Then $\{\X^{(\gamma)}\}_{\gamma\in\mathit{\Gamma}}$ is a complete family of non-trivial $(\C, \D)$-primitive matrices.
\end{proof}

The important property of a complete family of non-trivial $(\C, \D)$-primitive matrices in the vector space spanned by all $(\C, \D)$-primitive matrices is summarized in the following proposition.
\begin{proposition}\label{prop:Y=kx1+tx2}
Suppose $\{\X^{(\gamma)}\}_{\gamma\in\mathit{\Gamma}}$ is a complete family of non-trivial $(\C, \D)$-primitive matrices. Then for any $(\C, \D)$-primitive matrix $\Y$, we have $\overline{\Y}=\sum\limits_{\gamma\in \mathit{\Gamma}}\alpha_\gamma\overline{\X^{(\gamma)}}$, where $\alpha_\gamma\in\k \;(\gamma\in \mathit{\Gamma})$ and only a finite number of them are nonzero.
\end{proposition}
\begin{proof}
Suppose that $\X^{(\gamma)}=(x_{ij}^{(\gamma)})_{r\times s}$ and $\Y=(y_{ij})_{r\times s}$. By the definition of $(\C, \D)$-primitive matrix, we have
$$\Delta(y_{ij})=\sum\limits_{k=1}^rc_{ik}\otimes y_{kj}+\sum\limits_{l=1}^sy_{il}\otimes d_{lj}.$$
According to Corollary \ref{coro:basis}, for any $1\leq i\leq r, 1\leq j\leq s$, we can assume
$$\overline{y_{ij}}=\sum\limits_{\gamma\in \mathit{\Gamma}}\sum\limits_{p=1}^r\sum\limits_{q=1}^s \beta_{pq}^{(ij,\gamma)}\overline{x_{pq}^{(\gamma)}},$$
where $\beta_{pq}^{(ij,\gamma)}\in \k$ for any $1\leq p\leq r, 1\leq q \leq s, \gamma\in \mathit{\Gamma}$, and only a finite number of them are nonzero.
Then
\begin{eqnarray*}
\rho_L(\overline{y_{ij}})&=&(\id\otimes\pi)\Delta(y_{ij})=\sum\limits_{k=1}^rc_{ik}\otimes \overline{y_{kj}}\\
&=&\rho_L(\sum\limits_{\gamma\in \mathit{\Gamma}}\sum\limits_{p=1}^r\sum\limits_{q=1}^s \beta_{pq}^{(ij,\gamma)}\overline{x_{pq}^{(\gamma)}})\\
&=&\sum\limits_{p=1}^r\sum\limits_{k=1}^rc_{pk}\otimes(\sum\limits_{\gamma \in \mathit{\Gamma}} \sum\limits_{q=1}^s \beta_{pq}^{(ij,\gamma)}\overline{x_{kq}^{(\gamma)}}).
\end{eqnarray*}
Since the entries of $\C$ are linearly independent, it follows that $$\overline{y_{kj}}=\sum\limits_{\gamma \in \mathit{\Gamma}}\sum\limits_{q=1}^s \beta_{iq}^{(ij,\gamma)}\overline{x_{kq}^{(\gamma)}}$$
holds for any $1\leq k \leq r$.
Using the same argument as above, when we consider the right comodule structure of $\span(\overline{\Y})$, we get $$\overline{y_{il}}=\sum\limits_{\gamma \in \mathit{\Gamma}}\sum\limits_{p=1}^r \beta_{pj}^{(ij,\gamma)}\overline{x_{pl}^{(\gamma)}}$$
holds for any $1\leq l \leq s$.
It follows that $$\overline{y_{kl}}=\sum\limits_{\gamma \in \mathit{\Gamma}}\sum\limits_{q=1}^s \beta_{iq}^{(il,\gamma)}\overline{x_{kq}^{(\gamma)}}=\sum\limits_{\gamma \in \mathit{\Gamma}}\sum\limits_{p=1}^r \beta_{pj}^{(kj,\gamma)}\overline{x_{pl}^{(\gamma)}},$$
for any $1\leq k \leq r, 1\leq l \leq s$.
Because of the linear independence of $$\{\overline{x_{ij}^{(\gamma)}}\mid 1\leq i\leq r, 1\leq j\leq s, \gamma\in \mathit{\Gamma}\}$$ in $H_1/H_0$, for any $1\leq i,k\leq r, 1\leq j,l \leq s, \gamma\in \mathit{\Gamma},$
 we have$$\beta_{iq}^{(il,\gamma)}=0$$ when $q\neq l$, and
 $$\beta_{pj}^{(kj,\gamma)}=0$$when $p\neq k$. Moreover, when $p=k, q=l$, $$\beta_{il}^{(il,\gamma)}=\beta_{kj}^{(kj,\gamma)}$$ holds for all $1\leq i,k\leq r, 1\leq j,l \leq s, \gamma\in \mathit{\Gamma}$. This means that $$\overline{\Y}=\sum\limits_{\gamma\in \mathit{\Gamma}}\alpha_{\gamma} \overline{\X^{(\gamma)}},$$ where $\alpha_\gamma=\beta_{11}^{(11,\gamma)}\in\k$ for any $\gamma\in \mathit{\Gamma}$, and only a finite number of them are nonzero.
\end{proof}

With the help of the preceding proposition, we can now prove:
\begin{corollary} \label{lemma:PXQ=nX}
Let $\X$ be a non-trivial $(\C, \D)$-primitive matrix. Suppose $P\X Q$ is also a $(\C, \D)$-primitive matrix, where $P$ and $Q$ are invertible matrices over $\k$. Then $P\overline{\X} Q=\alpha \overline{\X}$ for some $\alpha\in\k$.
\end{corollary}
\begin{proof}
Using Lemma \ref{lemma:complement}, we can find a complete family of non-trivial $(\C, \D)$-primitive matrices
$\{\X^{(\gamma)}\}_{\gamma\in\mathit{\Gamma}}$ with some element $\X^{(\gamma_1)}=\X$. Then by Proposition \ref{prop:Y=kx1+tx2},
$$P\overline{\X}Q=\sum\limits_{\gamma\in\Gamma}\alpha_{\gamma}\overline{\X^{(\gamma)}},$$ where $\alpha_\gamma\in\k\;(\gamma\in \mathit{\Gamma})$ and only a finite number of them are nonzero. However,
$$\span(P\overline{\X }Q)\cap \big(\sum_{\gamma\in \Gamma\setminus \{\gamma_1\}}\span(\overline{\X^{(\gamma)}})\big)=0. $$  This implies that $\alpha_{\gamma}=0$ for all $\gamma\neq\gamma_1$.
Therefore, $$P\overline{\X} Q=\alpha_{\gamma_1} \overline{\X}.$$
\end{proof}

Note that by \cite[Theorem 4.1]{CHZ06}, we have $$H_1/H_0\cong \bigoplus_{C, D\in\mathcal{S}}(C\wedge D)/(C+D),$$
where $(C\wedge D)/(C+D)$ is isomorphic to the following $C$-$D$-bicomodule
$$\{h\in H_1/H_0\mid \rho_L(h)\in C\otimes H_1/H_0,\;\rho_R(h)\in  H_1/H_0\otimes D\},$$
which is exactly ${}^{C}(H_1/H_0){}^{D}$. So we can now obtain the following lemma:
\begin{lemma}\label{wedge=idempotent}
If $C, D\in\mathcal{S}$, then we have a $C$-$D$-bicomodule isomorphism: $$({}^CH_1{}^D+H_0)/H_0\cong (C\wedge D)/(C+D).$$
\end{lemma}

Combining Corollary \ref{coro:complete} and Lemma \ref{wedge=idempotent}, we obtain the following corollary.
\begin{corollary}\label{coro:complete number}
Let $C, D\in\mathcal{S}$ with basic multiplicative matrices $\C_{r\times r}$ and $\D_{s\times s}$, respectively.
If $\{\X^{(\gamma)}\}_{\gamma\in\mathit{\Gamma}}$ is a complete family of non-trivial $(\C, \D)$-primitive matrices, then the cardinal number
\begin{eqnarray}\label{arrownumber=primitive}
\mid \mathit{\Gamma}\mid=\frac{1}{rs}\dim_{\k}\left((C\wedge D)/(C+D)\right).
\end{eqnarray}
\end{corollary}
The corollary above will help us transform the problem of number of arrows from vertex $D$ to vertex $C$ in the link quiver of $H$ to the problem of cardinal number of a complete family of non-trivial $(\C, \D)$-primitive matrices in the subsequent sections.

Note that the number (\ref{arrownumber=primitive}) in Corollary \ref{coro:complete number} does not depend on the choices of basic multiplicative matrices $\C$ and $\D$ as well as a complete family of non-trivial $(\C, \D)$-primitive matrices.
\section{Constructions of a complete family of non-trivial primitive matrices}\label{section3}
Recall that a finite-dimensional Hopf algebra is said to have the dual Chevalley property, if its coradical $H_0$ is a Hopf subalgebra. In this paper, we still use the term \textit{dual Chevalley property} to indicate a Hopf algebra $H$ with its coradical $H_0$ as a Hopf subalgebra, even if $H$ is infinite-dimensional.

In this section, let $H$ be a Hopf algebra over $\k$ with the dual Chevalley property. Denote the coradical filtration of $H$ by $\{H_n\}_{n\geq0}$ and the set of all the simple subcoalgebras of $H$ by $\mathcal{S}$. We say that two matrices $\A$ and $\A^\prime$ over $H$ are \textit{similar}, which is denoted by $\A\sim\A^\prime$ for simplicity, if there exists an invertible matrix $L$ over $\k$ such that $\A^\prime=L\A L^{-1}$.

The aim of this section is to construct a complete family of non-trivial $(\C, \D)$-primitive matrices over $H$ for any $C, D\in\mathcal{S}$ with basic multiplicative matrices $\C, \D$, respectively.

\subsection{The first construction}
Let us start by recalling some notations.

For any matrix $\A=(a_{ij})_{r\times s}$ and $\B=(b_{ij})_{u\times v}$ over $H$, define $\A\odot \B$ and $\A\odot^\prime \B$ as follow
 $$\A\odot \B=
\left(\begin{array}{ccc}
      a_{11}\B& \cdots &  a_{1s}\B  \\
      \vdots  & \ddots & \vdots  \\
      a_{r1}\B&  \cdots & a_{rs}\B
    \end{array}\right),\;\;
\A\odot^\prime \B=\left(\begin{array}{cccc}
      \A b_{11} &   \cdots &  \A b_{1v} \\
      \vdots &  \ddots & \vdots  \\
      \A b_{u1} &  \cdots & \A b_{uv}
    \end{array}\right).$$

Some evident formulas on $\odot$ and $\odot^\prime$ should be noted for later computations.
\begin{lemma}\label{lemma:Kronecker}
Let $\A, \B$ be matrices over $H$ and $I$ be the identity matrix over $\k$, then
\begin{itemize}
  \item[(1)]  $(L_1\A L_2)\odot^\prime\B=(L_1\odot^\prime I)(\A\odot^\prime\B)(L_2\odot^\prime I)$ holds for any invertible matrices $L_1$, $L_2$ over $\k$;
  \item[(2)] $\A\odot(L_1\B L_2)=(I\odot L_1)(\A\odot\B)(I\odot L_2)$ holds for any invertible matrices $L_1$, $L_2$ over $\k$;
  \item[(3)] There exist invertible matrices $K, K^\prime$ over $\k$ such that $K(\A\odot \B)K^\prime=\A\odot^\prime \B$. Moreover, if $\A, \B$ are square matrices, then $\A\odot^\prime \B
      \sim \A\odot \B$.
\end{itemize}
\end{lemma}
\begin{proof}
\begin{itemize}
\item[(1)]Suppose that
$\B=\left(\begin{array}{ccc}
    b_{11}  & \cdots & b_{1v}  \\
    \vdots & & \vdots  \\
    b_{u1} &  \cdots & b_{uv}
  \end{array}\right),$
then
\begin{eqnarray*}
(L_1\A L_2)\odot^\prime\B&=&\left(\begin{array}{ccc}
    L_1\A L_2b_{11}  & \cdots & L_1\A L_2b_{1v}  \\
    \vdots  &  &\vdots  \\
    L_1\A L_2b_{u1}  & \cdots & L_1\A L_2b_{uv}
  \end{array}\right)\\
  &=&
 \left(\begin{array}{ccc}
    L_1\A b_{11}L_2  & \cdots & L_1\A b_{1v}L_2  \\
    \vdots &  &\vdots  \\
    L_1\A b_{u1}L_2  & \cdots & L_1\A b_{uv}L_2
  \end{array}\right)\\
&=&\left(\begin{array}{ccc}
    L_1  \\
    &   \ddots \\
    &&L_1
  \end{array}\right)
  \left(\begin{array}{ccc}
   \A b_{11} &  \cdots & \A b_{1v}  \\
    \vdots &   &\vdots  \\
  \A b{u1}  & \cdots & \A b_{uv}
  \end{array}\right)
  \left(\begin{array}{ccc}
    L_2  \\
    &  \ddots \\
    &&L_2
  \end{array}\right)
  \\
  &=&(L_1\odot^\prime I)(\A\odot^\prime\B)(L_2\odot^\prime I).
  \end{eqnarray*}
  \item[(2)]Consider the Hopf algebra $H^{\mathrm{op}}$, whose multiplication is opposite to $H$. Using (1), we can get this result.
  \item[(3)]By \cite[Theorem 8.26]{Sch17}, there exist commutation matrices $K, K^\prime$ such that $$K(\A\odot \B)K^\prime=\A\odot^\prime \B,$$where commutation matrix is defined in \cite[Definition 8.1]{Sch17}. Moreover, from the proof of \cite[Theorem 8.24]{Sch17}, we know that if $\A, \B$ are square matrices, then $$\A\odot^\prime \B
      \sim \A\odot \B.$$
  \end{itemize}
  \end{proof}
Let $B, C, D\in\mathcal{S}$ with basic multiplicative matrices $\B, \C, \D$ respectively.
According to \cite[Proposition 2.6]{Li22a}, there exists an invertible matrices $L_{\B, \C}$ over $\k$ such that
$$
L_{\B, \C}(\B\odot^\prime \C)L_{\B, \C}^{-1}=
\left(\begin{array}{ccc}
        \E_{1} &  &   \\
       & \ddots &   \\
        &  &   \E_{ u_{(\B, \C)}}
  \end{array}\right),
$$
where $\E_{1}, \cdots, \E_{ u_{(\B, \C)}}$ are the basic multiplicative matrices of $E_{1}, \cdots, E_{ u_{(\B, \C)}}$, respectively.
In particular, let $L_{1, \C}=L_{\C, 1}=I$, where $I$ is the identity matrix over $\k$.
Note that cosemisimple coalgebra $BC$ admits a decomposition into a direct sum of simple subcoalgebras and $u_{(\B, \C)}$ is exactly the number of such simple subcoalgebras. Thus in fact $u_{(\B, \C)}$ does not depend on the choices of basic multiplicative matrices $\B$ and $\C$ as well as the invertible matrix $L_{\B, \C}$.

For any $(\C, \D)$-primitive matrix $\X$, by \cite[Proposition 2.6]{Li22a}, there exist invertible matrices $L_{\B, \C}, L_{\B, \D}$ over $\k$ such that
\begin{eqnarray}\label{equation:BX}
&&\left(\begin{array}{cc}
    L_{\B, \C}  \\
    & L_{\B, \D}
  \end{array}\right)\left(\B\odot^\prime
\left(\begin{array}{cc}
\C&\X\\
0&\D
 \end{array}\right)
 \right)
 \left(\begin{array}{cc}
    L_{\B, \C}^{-1}  \\
    & L_{\B, \D}^{-1}
  \end{array}
 \right)\notag\\
&=&\left(\begin{array}{cc}
    L_{\B, \C}  \\
    & L_{\B, \D}
  \end{array}\right)
  \left(\begin{array}{cc}
\B\odot^\prime\C&\B\odot^\prime\X\\
0&\B\odot^\prime\D
 \end{array}
\right)
 \left(\begin{array}{cc}
    L_{\B, \C}^{-1}  \\
    & L_{\B, \D}^{-1}
  \end{array}
  \right)\notag\\
&=&\left(\begin{array}{cccccc}
    \E_1 &  &  & \X_{11} & \cdots & \X_{1u_{(\B, \D)}}  \\
    & \ddots &  & \vdots &  & \vdots  \\
     & & \E_{u_{(\B, \C)}} & \X_{u_{(\B, \C)}1} & \cdots & \X_{u_{(\B, \C)}u_{(\B, \D)}}  \\
     &  &  & \F_1 &  &   \\
     & 0 &  &  & \ddots &   \\
     &  &  &  &  & \F_{u_{(\B, \D)}}
  \end{array}\right),
\end{eqnarray}
 where $\E_1, \cdots, \E_{u_{(\B, \C)}}, \F_1, \cdots, \F_{u_{(\B, \D)}}$ are the given basic multiplicative matrices.
Combining \cite[Remark 2.5 and Lemma 2.7]{Li22a} and \cite[Remark 3.2]{LZ19}, we can show that each $\X_{ij}$ is a $(\E_i, \F_j)$-primitive matrix.

With the notations above, we have
\begin{lemma}\label{Lemma:CXno0}
For any $B, C, D\in\mathcal{S}$ with basic multiplicative matrices $\B, \C, \D$ respectively. If $\X$ is a non-trivial $(\C, \D)$-primitive matrix, then
\begin{itemize}
  \item[(1)]The set of all row vectors of $\B\odot^\prime \X$ is linearly independent over $H_1/H_0$;
  \item[(2)]The set of all column vectors of $\B\odot^\prime \X$ is linearly independent over $H_1/H_0$;
  \item[(3)]For each $1\leq i\leq u_{(\B, \C)}$, there is some $1\leq j\leq u_{(\B, \D)}$ such that $\X_{ij}$ is non-trivial;
  \item[(4)]For each $1\leq j\leq u_{(\B, \D)}$, there is some $1\leq i\leq u_{(\B, \C)}$ such that $\X_{ij}$ is non-trivial.
  \end{itemize}
\end{lemma}
\begin{proof}
These four claims are exactly (i), (ii), (I), (II) appearing in the proof of \cite[Lemma 3.12]{Li22a} in the case of $H^{\mathrm{op}}$.
\end{proof}

\begin{notation}
Let $\mathcal{M}$ denote the set of representative elements of basic multiplicative matrices over $H$ for the similarity class.
\end{notation}
\begin{remark}
It is clear that there is a bijection from $\mathcal{S}$ to $\mathcal{M}$, mapping each simple subcoalgebra to its basic multiplicative matrix, and $\mathcal{S}=\{\span(\C)\mid \C\in \mathcal{M}\}$, where $\span(\C)$ is the subspace of $H_0$ spanned by the entries of $\C$.
\end{remark}
Denote ${}^1\mathcal{S}=\{C\in\mathcal{S}\mid \k1+C\neq \k1\wedge C\}$. For any $C\in{}^1\mathcal{S}$ with basic multiplicative matrix $\C\in\mathcal{M}$, using Corollary \ref{coro:complete}, we can fix a complete family $\{\X_{C}^{(\gamma_C)}\}_{\mathit{\gamma_C\in\Gamma}_C}$ of non-trivial $(1, \C)$-primitive matrices.

Denote
\begin{eqnarray}\label{def:1^P}
{^1\mathcal{P}}:=\bigcup\limits_{C\in {}^1\mathcal{S}}\{\X_{C}^{(\gamma_C)}\mid \gamma_C\in \mathit{\Gamma}_C\}.
\end{eqnarray}
Then for any non-trivial $(1, \C)$-primitive matrix $\Y\in{^1\mathcal{P}}$ and $\B\in\mathcal{M}$, we have
\begin{eqnarray}\label{equationBY}
\left(\begin{array}{cc}
I&0\\
0&L_{\B, \C}
 \end{array}\right)
\left(\B\odot^\prime
\left(\begin{array}{cc}
1&\Y\\
0&\C
 \end{array}\right)\right)
 \left(\begin{array}{cc}
I&0\\
0&L_{\B, \C}^{-1}
 \end{array}\right)
=\left(\begin{array}{cccccc}
    \B&  &  & {\Y_{ 1}} & \cdots & {\Y_{ u_{(\B, \C)}}}  \\
     &  &  & \E_{1} &  &   \\
    0&  &   &  & \ddots &   \\
     &  &  &  &  & \E_{u_{(\B, \C)}}
  \end{array}\right),
\end{eqnarray}
where $\E_1, \E_2, \cdots, \E_{u_{(\B, \C)}}\in \mathcal{M}$. According to Lemma \ref{Lemma:CXno0}, we know that $\Y_1, \Y_2, \cdots, \Y_{u_{(\B, \C)}}$ are non-trivial.\\
Denote
\begin{eqnarray}\label{^BPY}
^{\B}\mathcal{P}_{\Y}:=\{\Y_{ i}\mid 1\leq i\leq u_{(\B, \C)}\},
\end{eqnarray}
\begin{eqnarray}\label{definition:^BP}
^{\B}\mathcal{P}:=\bigcup\limits_{\Y\in{{}^1\mathcal{P}}}{}^{\B}\mathcal{P}_{\Y},\;\;\; \mathcal{P}_{\Y}:=\bigcup\limits_{\B\in \mathcal{M}}{}^{\B}\mathcal{P}_{\Y}.
\end{eqnarray}
We remark that $\bigcup\limits_{\Y\in{{}^1\mathcal{P}}}{}^{1}\mathcal{P}_{\Y}$ coincides with ${}^1\mathcal{P}$ defined in (\ref{def:1^P}).\\
Moreover, denote
\begin{eqnarray}\label{definition:P}
\mathcal{P}:=\bigcup\limits_{\B\in \mathcal{M}}{^{\B}\mathcal{P}}=\bigcup\limits_{\Y\in{{}^1\mathcal{P}}}\mathcal{P}_{\Y} .
\end{eqnarray}
Note that the elements in the set $^{\B}\mathcal{P}_{\Y}$ depends on the choice of the invertible matrix $L_{\B, \C}$ in (\ref{equationBY}). It will be shown in the following Lemma that the cardinal number $\mid{}^\B\mathcal{P}_{\Y}\mid$ does not depend on the choice of $L_{\B, \C}$.

\begin{lemma}\label{Lemma:cap=0}
The sum $\sum\limits_{1\leq i\leq u_{(\B, \C)}}\span(\overline{\Y_{ i}})$ is direct, where each $\Y_{ i}$ appears in (\ref{equationBY}).
\end{lemma}
\begin{proof}
Without loss of generality, assume that
$$E_1=E_2=\cdots=E_t$$ for some $1\leq t\leq u_{\B, \C}$, and
$$E_j\neq E_1$$ when $t< j\leq u_{(\B, \C)}$.
In fact $\sum\limits_{1\leq i\leq t }\span(\overline{\Y_{ i}})$ is a $B$-$E_1$-sub-bicomodule of $H_1/H_0$. Let $T_0$ be a maximal subset of $\{1, 2, \cdots, t\}$ such that $\sum\limits_{j\in T_0} \span(\overline{\Y_{ j}})$ is direct. Suppose $$\bigoplus\limits_{j\in T_0} \span(\overline{\Y_{ j}})\subsetneqq \sum\limits_{1\leq i\leq t}\span(\overline{\Y_{ i}}).$$ Since for any $1\leq i\leq t$, $\span(\overline{\Y_{ i}})$ is simple, there exists some $s\notin T_0$, $$\span(\overline{\Y_{ s}})\cap (\bigoplus\limits_{j\in T_0} \span(\overline{\Y_{ j}}))=0.$$ Thus $$\span(\overline{\Y_{ s}})+ (\bigoplus\limits_{j\in T_0} \span(\overline{\Y_{ j}}))=\span(\overline{\Y_{ s}})\oplus (\bigoplus\limits_{j\in T_0} \span(\overline{\Y_{ j}})),$$ which is a contradiction. Now we can get a subset $\{w_1, \cdots, w_r\} $ of $\{1, \cdots, t\}$ such that $$\sum\limits_{i=1 }^t\span(\overline{\Y_{ i}})=\bigoplus_{i=j}^r\span(\overline{\Y_{ w_j}}).$$
Without loss of generality, assume that $$\{w_1, w_2, \cdots, w_r\}=\{1, 2, \cdots, r\}.$$
According to Lemma \ref{lemma:complement}, there exists a complete family $\{\X^{(\gamma)}\}_{\gamma\in \mathit{\Gamma}}$ of non-trivial $(\B, \E_{ 1})$-primitive matrices such that $\{\Y_{ i}\}_{1\leq i\leq r}$ is a subset of $\{\X^{(\gamma)}\}_{\gamma\in \mathit{\Gamma}}$. It follows from Proposition \ref{prop:Y=kx1+tx2} that $\overline{\Y_{ t}}$ is the linear combination of $\{\overline{\X^{(\gamma)}}\}_{\gamma\in \mathit{\Gamma}}$. Note that if $t>r$, then $$\span(\overline{\Y_{ t}})\subseteq\bigoplus_{1\leq i\leq r}\span(\overline{\Y_{ i}}).$$ According to Corollary \ref{coro:basis}, $\overline{\Y_t}$ is the linear combination of $\{\overline{\Y_i}\}_{1\leq i\leq r}$. This implies that the column vectors of
$\left(\begin{array}{cccccc}
{\Y_{ 1}} & {\Y_{ 2}} &\cdots & {\Y_{ t}}
\end{array}\right)$
are linearly dependent over $H/H_0$, which is in contradiction with Lemma \ref{Lemma:CXno0}. Thus we have $t=r$ and the sum $\sum\limits_{1\leq i\leq t }\span(\overline{\Y_{ i}})$ is direct. Then by Corollary \ref{coro:complete}, the proof is completed.
\end{proof}

\begin{remark}\label{remark:ubc}
The cardinal number $\mid {}^{\B}\mathcal{P}_{\Y}\mid=u_{B, C}$, where ${}^{\B}\mathcal{P}_{\Y}$ appears in (\ref{^BPY}).
\end{remark}
Now we define an $H_0$-bimodule structure on $H_1/H_0$ as follows:
$$h\otimes \overline{x}\mapsto h\cdot \overline{x}:=\overline{hx},\; \;\overline{x}\otimes h\mapsto  \overline{x}\cdot h:=\overline{xh} \;(h\in H_0,x\in H_1).$$
Thus $H_1/H_0$ becomes an $H_0$-Hopf bimodule with the bicomodule structure defined in (\ref{comodulestructure}) and bimodule structure defined obove.
\begin{lemma}\label{Lem:P=H_1/H_0}
With the notations in (\ref{definition:P}), we have $H_1/H_0=\sum\limits_{\X\in\mathcal{P}}\span(\overline{\X})$.
\end{lemma}
\begin{proof}
It suffices us to prove that $$H_1/H_0\subseteq\sum\limits_{\X\in\mathcal{P}}\span(\overline{\X}).$$
Applying the fundamental theorem of Hopf modules (\cite[Theorem 4.1.1]{Swe69}), we know that as a left $H_0$-Hopf module, $$ H_0\otimes {^{coH_0}(H_1/H_0)}\cong H_1/H_0,$$
where ${(H_1/H_0){}^{coH_0}}$ is the left coinvariants of $H_0$ in $H_1/H_0$.
This isomorphism maps $h\otimes \overline{x}$ to $ \overline{hx}$, where $h\in H_0, \overline{x}\in  {^{coH_0}(H_1/H_0)}$. Using \cite[Proposition 2.6(4)]{LL22}, we can obtain the direct sum decomposition $${^1H_1}=\bigoplus_{C\in\mathcal{S}}{^1{H_1}^C}.$$
It follows that
\begin{eqnarray*}
({}^1{H_1}+H_0)/H_0&=&\big((\bigoplus_{C\in\mathcal{S}}{^1{H_1}^C})+H_0\big)/H_0\\
&=&(\sum_{C\in\mathcal{S}}{^1{H_1}^C}+H_0)/H_0\\
&=&\sum_{C\in\mathcal{S}}({^1{H_1}^C}+H_0)/H_0.
\end{eqnarray*}
Note that $\sum_{C\in\mathcal{S}}({^1{H_1}^C}+H_0)/H_0$ is direct and according to Corollary \ref{coro:complete}, we have
\begin{eqnarray}\label{3.7}
({}^1{H_1}+H_0)/H_0=\bigoplus_{C\in\mathcal{S}}({^1{H_1}^C}+H_0)/H_0=\bigoplus_{\Y\in{}^1\mathcal{P}}\span(\overline{\Y}).
\end{eqnarray}
From the proof of \cite[Proposition 3.9]{LL22}, we know that
\begin{eqnarray}\label{3.8}
(H_1/H_0){}^{coH_0}=\bigoplus_{\Y\in{}^1\mathcal{P}}\span(\overline{\Y}).
\end{eqnarray}
Moreover, the definition of $\mathcal{P}$ yields that
$$\sum\limits_{\X\in\mathcal{P}}\span(\overline{\X})=\sum_{\B\in \mathcal{M}}\sum_{\Y\in{}^1\mathcal{P}}\span(\overline{\B\odot^\prime \Y}).$$
In fact
\begin{eqnarray}\label{3.9}
H_0=\sum_{\B\in\mathcal{M}}\span(\B).
 \end{eqnarray}
According to (\ref{3.7}) and (\ref{3.9}), one can get
\begin{eqnarray*}
H_1/H_0&=& \{h\cdot\overline{x} \mid h\in H_0, x\in {}^1{H_1}+H_0\}\\
&\subseteq& \big(\sum_{\B\in\mathcal{M}}\span(\B)\big)\cdot \big(\bigoplus_{\Y\in{}^1\mathcal{P}}\span(\overline{\Y})\big)\\
&\subseteq&\sum_{\B\in \mathcal{M}}\sum_{\Y\in{}^1\mathcal{P}}\span(\overline{\B\odot^\prime \Y})\\
&=&\sum\limits_{\X\in\mathcal{P}}\span(\overline{\X}).
\end{eqnarray*}
\end{proof}
\begin{lemma}\label{lemma:PXnoPY}
With the notations in (\ref{definition:^BP}), for any $\C\in\mathcal{M}$ and non-trivial $(1, \C)$-primitive matrix $\X\in{}^1\mathcal{P}$, we have
$$\left( \sum_{\W\in \mathcal{P}_{\X}}\span(\overline{\W})\right)\cap \left( \sum\limits_{\Z\in {}^1\mathcal{P}, \Z\neq \X} \sum_{\W\in \mathcal{P}_{\Z}}\span(\overline{\W})\right)=0.$$
\end{lemma}

\begin{proof}
According to \ref{3.8}, we have $$H_0\otimes(H_1/H_0){}^{coH_0}=\bigoplus\limits_{\Y\in{{^1\mathcal{P}}}}H_0\otimes \span(\overline{\Y}).$$
Moreover, the isomorphism $$ H_0\otimes (H_1/H_0){}^{coH_0}\cong H_1/H_0$$  maps $h\otimes \overline{x}$ to $h\cdot \overline{x}$, where $h\in H_0, \overline{x}\in  {^{coH_0}(H_1/H_0)}$.
By the definition of ${}^1\mathcal{P}$, we know that $$\span(\overline{\X})\cap\big(\sum\limits_{\Z\in {}^1\mathcal{P}, \Z\neq \X}\span(\overline{\Z})\big)=0.$$
Then $$\left(H_0\otimes\span(\overline{\Y})\right)\cap \left(H_0\otimes\big(\sum\limits_{\Z\in {}^1\mathcal{P}, \Z\neq \X}\span(\overline{\Z})\big)\right)=0,$$
which suggests that
$$\left(H_0\cdot\span(\overline{\Y})\right)\cap \left(H_0\cdot\big(\sum\limits_{\Z\in {}^1\mathcal{P}, \Z\neq \X}\span(\overline{\Z})\big)\right)=0,$$
where
\begin{eqnarray*}
H_0\cdot\span(\overline{\Y})&=&\{h\cdot \overline{y}\mid h\in H_0, y\in \Y\},\\
H_0\cdot\big(\sum\limits_{\Z\in {}^1\mathcal{P}, \Z\neq \X}\span(\overline{\Z})\big)&=&\{h\cdot \overline{z}\mid h\in H_0, \overline{z}\in \sum\limits_{\Z\in {}^1\mathcal{P}, \Z\neq \X}\span(\overline{\Z})\}.
\end{eqnarray*}
Therefore, by the definition of $\mathcal{P}_{\X}$, we conclude that
$$\left( \sum_{\W\in \mathcal{P}_{\X}}\span(\overline{\W})\right)\cap \left( \sum\limits_{\Z\in {}^1\mathcal{P}, \Z\neq \X} \sum_{\W\in \mathcal{P}_{\Z}}\span(\overline{\W})\right)=0.$$
\end{proof}
A direct consequence of this lemma is:
\begin{corollary}\label{coro:P_X}
With the notations in (\ref{definition:P}), then the union $\mathcal{P}=\bigcup\limits_{\Y\in{}^1\mathcal{P}}\mathcal{P}_{\Y}$ is disjoint.
\end{corollary}
Now it is not difficult to verify the following theorem.
\begin{theorem}\label{coro:BXcomplete}
Let $C, D\in \mathcal{S}$ with basic multiplicative matrices $\C, \D\in\mathcal{M}$ respectively. Denote
$${}^{\C}\mathcal{P}^{\D}:=\{\X\in \mathcal{P}\mid  \X \text{ is a non-trivial }(\C, \D)\text{-primitive matrix}\}.$$
Then it is a complete family of non-trivial $(\C, \D)$-primitive matrices. Moreover, we have $H_1/H_0=\bigoplus_{\X\in\mathcal{P}}\span(\overline{\X})$.
\end{theorem}
\begin{proof}
By the definition of $\mathcal{P}$, we have $$\mathcal{P}=\bigcup_{\C, \D\in\mathcal{M}} {}^{\C}\mathcal{P}^{\D},$$
which means that
$$\sum_{\X\in\mathcal{P}}\span(\overline{\X})=\sum_{\C,\D\in\mathcal{M}}\sum_{\X\in{}^{\C}\mathcal{P}^{\D}}\span(\overline{\X}).$$
According to Lemma \ref{lemma:sum2}, we know that $$H_1/H_0=\bigoplus_{C,D\in\mathcal{S}}({}^CH_1{}^{D}+H_0)/H_0.$$
It follows from Lemma \ref{Lemma:XinCPD} that $$\sum_{\X\in{}^{\C}\mathcal{P}^{\D}}\span(\overline{\X})\subseteq ({}^CH_1{}^{D}+H_0)/H_0,$$
which implies that
\begin{eqnarray*}
\sum_{\X\in\mathcal{P}}\span(\overline{\X})&=&\bigoplus_{\C,\D\in\mathcal{M}}\big(\sum_{\X\in{}^{\C}\mathcal{P}^{\D}}\span(\overline{\X})\big)\\
&\subseteq& \bigoplus_{C,D\in\mathcal{S}}({}^CH_1{}^{D}+H_0)/H_0\\
&=&H_1/H_0.
\end{eqnarray*}
By Lemma \ref{Lem:P=H_1/H_0}, we have $$H_1/H_0=\sum_{\X\in\mathcal{P}}\span(\overline{\X}).$$
Therefore one can get
\begin{eqnarray}\label{cpd=sumx}
({}^CH_1{}^{D}+H_0)/H_0=\sum_{\X\in{}^{\C}\mathcal{P}^{\D}}\span(\overline{\X}).
\end{eqnarray}
Note that $${}^{\C}\mathcal{P}_{\Y}\subseteq \mathcal{P}_{\Y}.$$
Combining Lemmas \ref{Lemma:cap=0} and \ref{lemma:PXnoPY}, we can get
\begin{eqnarray*}
\sum_{\Y\in{}^1\mathcal{P}}\big( \sum_{\X\in  {}^{\C}\mathcal{P}_{\Y}}\span(\overline{\X})\big)&=&
\sum_{\Y\in{}^1\mathcal{P}}\big( \bigoplus_{\X\in  {}^{\C}\mathcal{P}_{\Y}}\span(\overline{\X})\big)\\
&=&\bigoplus_{\Y\in{}^1\mathcal{P}}\big( \bigoplus_{\X\in  {}^{\C}\mathcal{P}_{\Y}}\span(\overline{\X})\big).
\end{eqnarray*}
Thus it follows from $${}^{\C}\mathcal{P}^{\D}\subseteq \bigcup_{\Y\in{}^1\mathcal{P}} {}^{\C}\mathcal{P}_{\Y}$$  that$$\sum_{\X\in  {}^{\C}\mathcal{P}^{\D}}\span(\overline{\X})=\bigoplus_{\X\in  {}^{\C}\mathcal{P}^{\D}}\span(\overline{\X}),$$ and ${}^{\C}\mathcal{P}^{\D}$ is a family of non-trivial $(\C, \D)$-primitive matrices.
Moreover, it is straightforward to show that  $$H_1/H_0=\bigoplus_{\X\in\mathcal{P}}\span(\overline{\X}).$$
\end{proof}
\begin{corollary}\label{coro:CP,PD}
Let $C, D\in\mathcal{S}$ with basic multiplicative matrices $\C, \D\in\mathcal{M}$ respectively. Denote
\begin{eqnarray*}
\mathcal{P}^{\D}:=\bigcup_{\C\in\mathcal{M}} {}^{\C}\mathcal{P}^{\D},
\end{eqnarray*}
which is a disjoint union of ${}^{\C}\mathcal{P}^{\D}$ $(\C\in\mathcal{M})$ defined in Theorem \ref{coro:BXcomplete}.
Then $$ (H_1{}^{D}+H_0)/H_0=\bigoplus_{\X\in \mathcal{P}^{\D}} \span(\overline{\X}).$$
Moreover, we have $${}^{\C}\mathcal{P}=\bigcup_{\D\in\mathcal{M}}{}^{\C}\mathcal{P}^{\D},$$
which is also a disjoint union, and that $$({}^{C}H_1+H_0)/H_0= \bigoplus_{\X\in {}^{\C}\mathcal{P}} \span(\overline{\X}).$$
\end{corollary}
\begin{proof}
Note that by the definition of ${}^{\C}\mathcal{P}{}^{\D}$, we know that ${}^{\C}\mathcal{P}{}^{\D}$ contains all the non-trivial $(\C, \D)$-primitive matrices in $\mathcal{P}$. It follows that $${}^{\C}\mathcal{P}=\bigcup_{\D\in\mathcal{M}}{}^{\C}\mathcal{P}^{\D}.$$
According to Lemma \ref{lemma:sum2}, we have $$H_1/H_0=\bigoplus_{C,D\in\mathcal{S}}({^C{H_1}^D}+H_0)/H_0.$$
By (\ref{cpd=sumx}) in the proof of Theorem \ref{coro:BXcomplete}, we know that $$({}^CH_1{}^{D}+H_0)/H_0=\sum_{\X\in{}^{\C}\mathcal{P}^{\D}}\span(\overline{\X}).$$
This means that these two unions $$\mathcal{P}^{\D}=\bigcup_{\C\in\mathcal{M}} {}^{\C}\mathcal{P}^{\D}$$ and
$${}^{\C}\mathcal{P}=\bigcup_{\D\in\mathcal{M}}{}^{\C}\mathcal{P}^{\D}$$ are both disjoint union.\\
Using \cite[Proposition 2.6(4)]{LL22}, we can obtain the direct sum decomposition $${H_1{}^{D}}=\bigoplus_{C\in\mathcal{S}}{^C{H_1}^D}.$$
It follows that
\begin{eqnarray*}
({H_1}{}^D+H_0)/H_0&=&\big((\bigoplus_{C\in\mathcal{S}}{^C{H_1}^D})+H_0\big)/H_0\\
&=&(\sum_{C\in\mathcal{S}}{^C{H_1}^D}+H_0)/H_0\\
&=&\sum_{C\in\mathcal{S}}({^C{H_1}^D}+H_0)/H_0.
\end{eqnarray*}
Then it follows from (\ref{cpd=sumx}) in the proof of Theorem \ref{coro:BXcomplete} that
\begin{eqnarray*}
(H_1{}^{D}+H_0)/H_0&=&\sum_{C\in\mathcal{S}} \sum_{\X\in{}^{\C}\mathcal{P}^{\D}}\span(\overline{\X})\\
&=&\sum_{\X\in \mathcal{P}^{\D}} \span(\overline{\X})\\
&=&\bigoplus_{\X\in \mathcal{P}^{\D}} \span(\overline{\X}),
\end{eqnarray*}
the last equation is due to  Theorem \ref{coro:BXcomplete}.
The proof of $$({}^{C}H_1+H_0)/H_0= \bigoplus_{\X\in {}^{\C}\mathcal{P}} \span(\overline{\X}) $$ is similar.
\end{proof}
\subsection{The second construction}
In this subsection, for any $C, D\in \mathcal{S}$, we give another construction of a complete family of non-trivial $(\C, \D)$-primitive matrices over $H$. We construct a set $\mathcal{P}^\prime$ in the following way.
Denote $\mathcal{S}^1=\{C\in\mathcal{S}\mid C+\k1\neq C\wedge \k1\}$. For any $C\in\mathcal{S}^1$ with basic multiplicative matrix $\C\in\mathcal{M}$, according to Corollary \ref{coro:complete}, we can fix a complete family $\{\X_{C}^{\prime(\gamma^\prime_C)}\}_{\mathit{\gamma^\prime_C\in\Gamma}^\prime_C}$ of non-trivial $(\C, 1)$-primitive matrices.

Denote
\begin{eqnarray}\label{def:1^Pprime}
{\mathcal{P}^{\prime 1}}=\bigcup\limits_{C\in \mathcal{S}_1}\{\X_{C}^{\prime(\gamma^\prime_C)}\mid \gamma^\prime_C\in \mathit{\Gamma}^\prime_C\}.
\end{eqnarray}
Then for any non-trivial $(\C, 1)$-primitive matrix $\Y^\prime\in{\mathcal{P}^{\prime 1}}$, we have
\begin{eqnarray}\label{equationBYprime}
\left(\begin{array}{cc}
L_{\B, \C}&0\\
0&I
 \end{array}\right)
\left(\B\odot^\prime
\left(\begin{array}{cc}
\C&\Y^\prime\\
0&1
 \end{array}\right)\right)
 \left(\begin{array}{cc}
L_{\B, \C}^{-1}&0\\
0&I
 \end{array}\right)
=\left(\begin{array}{cccccc}
      \E_{1} &  & & {\Y^\prime_{1}} \\
        & \ddots & &\vdots  \\
      &  & \E_{u_{(\B, \C)}}&{\Y^\prime_{u_{(\B, \C)}}}\\
     &0&&\B
  \end{array}\right),
\end{eqnarray}
where $I$ is the identity matrix over $\k$ and $\E_1, \E_2, \cdots, \E_{u_{(\B, \C)}}\in \mathcal{M}$.\\
Denote
\begin{eqnarray}\label{^BPYprime}
\mathcal{P}^{\prime \B}_{\Y^\prime}&=&\{\Y^\prime_{i}\mid 1\leq i\leq u_{(\B, \C)}\},
\end{eqnarray}
\begin{eqnarray}\label{definition:^BPprime}
\mathcal{P}^{\prime \B}&=&\bigcup\limits_{\Y^\prime\in{\mathcal{P}^{\prime 1}}}{}\mathcal{P}^{\prime \B}_{\Y^\prime},\;\;\; \mathcal{P}^\prime_{\Y^\prime}=\bigcup\limits_{\B\in \mathcal{M}}\mathcal{P}^{\prime \B}_{\Y^\prime},
\end{eqnarray}
and
\begin{eqnarray}\label{definition:Pprime}
\mathcal{P}^\prime&=&\bigcup\limits_{\B\in \mathcal{M}}{\mathcal{P}^{\prime \B}}=\bigcup\limits_{\Y^\prime\in{\mathcal{P}^1}}\mathcal{P}^\prime_{\Y^\prime} .
\end{eqnarray}
The same proof with Remark \ref{remark:ubc} can be applied to $\mid \mathcal{P}_{\Y^\prime}^{\prime\B}\mid$.
\begin{remark}\label{remark:ubcprime}
The cardinal number $\mid \mathcal{P}_{\Y^\prime}^{\prime\B}\mid=u_{B, C}$, where $\mathcal{P}^{\prime \B}_{\Y^\prime}$ appears in (\ref{^BPYprime}).
\end{remark}
According to \cite[Corollary 3.6]{Rad77}, since $H$ has the dual Chevalley property, the antipode $S$ of $H$ is bijective. Then for the mixed Hopf module $H_1/H_0$ in ${}_H\mathcal{M}^H$, we have $$ H_0\otimes {(H_1/H_0)^{coH_0}}\cong H_1/H_0,$$
where ${(H_1/H_0)^{coH_0}}$ is the right coinvariants of $H_0$ in $H_1/H_0$.
And the isomorphism maps $h\otimes \overline{x}$ to $h\cdot \overline{x}$, where $h\in H_0, \overline{x}\in  {(H_1/H_0)^{coH_0}}$.

The proofs of the following theorem and corollary can be completed by the method analogous to that used in the proofs of Theorem \ref{coro:BXcomplete} and Corollary \ref{coro:CP,PD}.
\begin{theorem}\label{coro:BXcompleteprime}
Let $C, D\in \mathcal{S}$ with basic multiplicative matrices $\C, \D\in\mathcal{M}$ respectively. Denote
$${}^{\C}\mathcal{P}^{\prime \D}:=\{\X^\prime\in \mathcal{P}^\prime\mid\X^\prime \text{ is a non-trivial }(\C, \D)\text{-primitive matrix}\}.$$
Then it is a complete family of non-trivial $(\C, \D)$-primitive matrices. Moreover, we have
$H_1/H_0=\bigoplus_{\X^{\prime}\in\mathcal{P}^{\prime}}\span(\overline{\X^{\prime}})$.
\end{theorem}
\begin{corollary}\label{coro:CP^prime,PD}
Let $C, D\in\mathcal{S}$ with basic multiplicative matrices $\C, \D\in\mathcal{M}$ respectively. Denote
\begin{eqnarray*}
{}^{\C}\mathcal{P}^{\prime}:=\bigcup_{\D\in\mathcal{M}}{}^{\C}\mathcal{P}^{\prime \D},
\end{eqnarray*}
which is a disjoint union of ${}^{\C}\mathcal{P}^{\prime \D}$ $(\D\in\mathcal{M})$ defined in Theorem \ref{coro:BXcomplete}.
Then $$({}^{C}H_1+H_0)/H_0= \bigoplus_{\X\in {}^{\C}\mathcal{P}^\prime} \span(\overline{\X}).$$
Moreover, we have \begin{eqnarray*}
\mathcal{P}^{\prime \D}=\bigcup_{\C\in\mathcal{M}} {}^{\C}\mathcal{P}^{\prime \D},
\end{eqnarray*}
which is also a disjoint union, and that $$ (H_1{}^{D}+H_0)/H_0=\bigoplus_{\X\in \mathcal{P}^{\prime \D}} \span(\overline{\X}).$$
\end{corollary}
So far, for any $C, D\in\mathcal{S}$ with $\dim_{\k}(C)=r^2, \dim(D)_{\k}=s^2$, we have already constructed two complete families of non-trivial $(\C, \D)$-primitive matrices over $H$. According to Corollary \ref{coro:complete number}, the cardinal number $\mid{}^{\C}\mathcal{P}^{\D}\mid=\mid{}^{\C}\mathcal{P}^{\prime \D}\mid=\frac{1}{rs}\dim_{\k}\left((C\wedge D)/(C+D)\right)$. Thus we can determine the number $\frac{1}{rs}\dim_{\k}\left((C\wedge D)/(C+D)\right)$ by studying ${}^{\C}\mathcal{P}^{\D}$ and ${}^{\C}\mathcal{P}^{\prime \D}$.
\section{Link quiver}\label{section4}
\subsection{$\Bbb{Z}_+$-rings}
Let $H$ be a Hopf algebra over $\k$ with the dual Chevalley property. For convenience, we still use the notations in Section \ref{section3}.

Let $\Bbb{Z}\mathcal{S}$ be the free additive abelian group generated by the elements of $\mathcal{S}$. For our purpose, let us start by giving a unital based $\Bbb{Z}_+$-ring structure on $\Bbb{Z}\mathcal{S}$. The related definitions and properties of $\Bbb{Z}_+$rings can be found in \cite[Section 2]{Ost03} and \cite[Chapter 3]{EGNO15}.

Let $\Bbb{Z}_+$ be the set of nonnegative integers. Some relevant concepts and results are recalled as follows.
\begin{definition}\emph{(}\cite[Definitions 2.1 and 2.2]{Ost03}\emph{)}
Let $A$ be an associative ring with unit which is free as a $\Bbb{Z}$-module.
\begin{itemize}
  \item[(1)]A $\Bbb{Z}_+$-basis of $A$ is a basis $B=\{b_{i}\}_{i\in I}$ such that $b_ib_j=\sum_{t\in I}c_{ij}^tb_t$, where $c_{ij}^t\in\Bbb{Z}_+$.
  \item[(2)]A ring with a fixed $\Bbb{Z}_+$-basis $\{b_i\}_{i\in I}$ is called a unital based ring if the following conditions hold:
  \begin{itemize}
  \item[(i)]$1$ is a basis element.
  \item[(ii)]Let $\tau: A\rightarrow \Bbb{Z}$ denote the group homomorphism defined by
  $$\tau(b_i)=\left\{
\begin{aligned}
1,~~~  \text{if} ~~~ b_i=1, \\
0,~~~  \text{if} ~~~ b_i\neq1.
\end{aligned}
\right.$$
There exists an involution $i \mapsto i^*$ of $I$ such that the induced map
$$a=\sum\limits_{i\in I}a_ib_i \mapsto a^*=\sum\limits_{i\in I}a_ib_{i^*},\;\; a_i\in \Bbb{Z}$$ is an anti-involution of $A$, and
$$\tau(b_ib_j)=\left\{
\begin{aligned}
1,~~~  \text{if} ~~~ i=j^*, \\
0,~~~  \text{if} ~~~ i\neq j^*.
\end{aligned}
\right.$$
  \end{itemize}
  \item[(3)]A fusion ring is a unital based ring of finite rank.
\end{itemize}
\end{definition}
It is straightforward to show the following lemma.
\begin{lemma}\emph{(}cf. \cite[Exercise 3.3.2]{EGNO15}\emph{)}\label{lemma:transitive}
Suppose $A$ is a unital based ring with $\Bbb{Z}_+$-basis $I$, then for any $X, Z\in I$, there exist $Y_1, Y_2$ such that $XY_1$ and $Y_1Z$ contain $Z$ with a nonzero coefficient.
\end{lemma}
\begin{proof}
Since both $XX^*$ and $Z^*Z$ contain $1$ with a nonzero coefficient, we can take $Y_1$ to be a suitable summand of $X^*Z$ and $Y_2$ to be a suitable summand of $ZX^*$.
\end{proof}
For any $B, C\in\mathcal{S}$ with basic multiplicative matrices $\B, \C \in\mathcal{M}$ respectively.
Since $H$ has the dual Chevalley property, it follows from
\cite[Proposition 2.6(2)]{Li22a} that there exists an invertible matrix $L$ over $\k$ such that
\begin{equation}\label{equationCD}
L
(\B\odot^{\prime}\C) L^{-1}=
\left(\begin{array}{cccc}
      \E_1 & 0 & \cdots & 0  \\
      0 & \E_2 & \cdots & 0  \\
      \vdots & \vdots & \ddots & \vdots  \\
      0 & 0 & \cdots & \E_t
    \end{array}\right),
    \end{equation}
where $\E_1, \E_2, \cdots, \E_t$ are basic multiplicative matrices over $H$.\\

Define a multiplication on $\Bbb{Z}\mathcal{S}$ as follow: for $B, C\in \mathcal{S}$,
$$B\cdot C=\sum\limits_{i=1}^t E_i,$$
where $E_1, \cdots, E_t\in\mathcal{S}$ are well-defined with basic multiplicative matrices $\E_i\in\mathcal{M}$ as in (\ref{equationCD}).

With the multiplication defined above, now we can prove the following proposition by using Lemma \ref{lemma:Kronecker}.
\begin{proposition}\label{Prop:basedring}
Let $H$ be a Hopf algebra over $\k$ with the dual Chevalley property and $\mathcal{S}$ be the set of all the simple subcoalgebras of $H$. Then $\Bbb{Z}\mathcal{S}$ is a unital based ring with $\Bbb{Z}_+$-basis $\mathcal{S}$.
\end{proposition}

\begin{proof}
For any $B, C, D\in\mathcal{S}$ with basic multiplicative matrices $\B, \C, \D$ respectively.
With the notations in (\ref{equationCD}), we know that $$(B\cdot C)\cdot D=\sum\limits_{i=1}^t E_i\cdot D.$$
By Lemma \ref{lemma:Kronecker}, we have
\begin{eqnarray*}
\left(\begin{array}{cccc}
      \E_1\odot^\prime \D & 0 & \cdots & 0  \\
      0 & \E_2\odot^\prime \D & \cdots & 0  \\
      \vdots & \vdots & \ddots & \vdots  \\
      0 & 0 & \cdots & \E_t\odot^\prime \D
    \end{array}\right)
    &\sim&
    \left(\begin{array}{cccc}
      \E_1\odot \D & 0 & \cdots & 0  \\
      0 & \E_2\odot \D & \cdots & 0  \\
      \vdots & \vdots & \ddots & \vdots  \\
      0 & 0 & \cdots & \E_t\odot \D
    \end{array}\right)\\
    &=&(L(\B\odot^\prime \C)L^{-1})\odot \D\\
    &\sim&(L(\B\odot^\prime \C)L^{-1})\odot^\prime \D\\
    &\sim&(\B\odot^\prime \C) \odot^\prime \D.
\end{eqnarray*}
Suppose that $$C\cdot D=\sum\limits_{i=1}^s F_i,$$ where $F_i\in\mathcal{S}$ for any $1\leq i\leq s,$ which means that
$$
L^\prime
(\C\odot^{\prime}\D) L^{\prime -1}=
\left(\begin{array}{cccc}
      \F_1 & 0 & \cdots & 0  \\
      0 & \F_2 & \cdots & 0  \\
      \vdots & \vdots & \ddots & \vdots  \\
      0 & 0 & \cdots & \F_s
    \end{array}\right)
$$
for some invertible matrix $L^\prime $ over $\k$.
Then $$B\cdot(C\cdot D)=\sum\limits_{i=1}^s B\cdot F_i.$$
By Lemma \ref{lemma:Kronecker}, we have
\begin{eqnarray*}
\left(\begin{array}{cccc}
      \B\odot^\prime\F_1 & 0 & \cdots & 0  \\
      0 &  \B\odot^\prime\F_2 & \cdots & 0  \\
      \vdots & \vdots & \ddots & \vdots  \\
      0 & 0 & \cdots &  \B\odot^\prime\F_s
    \end{array}\right)
&=&\B\odot^{\prime}
 \left(\begin{array}{cccc}
     \F_1 & 0 & \cdots & 0  \\
      0 &  \F_2 & \cdots & 0  \\
      \vdots & \vdots & \ddots & \vdots  \\
      0 & 0 & \cdots & \F_s
    \end{array}\right)\\
&=&\B\odot^\prime \left( L^\prime(\C\odot^{\prime}\D) L^{\prime -1} \right)\\
&\sim&\B\odot\left( L^\prime(\C\odot^{\prime}\D) L^{\prime -1} \right)\\
&\sim&\B\odot(\C\odot^\prime\D)\\
&\sim&\B\odot^\prime(\C\odot^\prime\D).
\end{eqnarray*}
As a conclusion, we have
$$\left(\begin{array}{cccc}
      \E_1\odot^\prime \D & 0 & \cdots & 0  \\
      0 & \E_2\odot^\prime \D & \cdots & 0  \\
      \vdots & \vdots & \ddots & \vdots  \\
      0 & 0 & \cdots & \E_t\odot^\prime \D
    \end{array}\right)
    \sim
    \left(\begin{array}{cccc}
      \B\odot^\prime\F_1 & 0 & \cdots & 0  \\
      0 &  \B\odot^\prime\F_2 & \cdots & 0  \\
      \vdots & \vdots & \ddots & \vdots  \\
      0 & 0 & \cdots &  \B\odot^\prime\F_s
    \end{array}\right).
$$
It follows that the traces of these two matrices are equal.
Thus a direct verification gives rise to the fact that $\Bbb{Z}\mathcal{S}$ is a unital $\Bbb{Z}_+$-ring.
Let $S$ be the antipode of $H$, then according to \cite[Theorem 3.3]{Lar71}, we get an anti-involution $C\mapsto S(C)$ of $\mathcal{S}$. It follows from \cite[Theorem 2.7]{Lar71} that there is only one $1$ in the summand of $C\cdot S(C)$, which means that
$\Bbb{Z}\mathcal{S}$ is a based ring.
\end{proof}
Note that if in addition $H_0$ is finite-dimensional, it is clear that $\Bbb{Z}\mathcal{S}$ is a fusion ring.
In this situation, we can study \textit{Frobenius-Perron dimensions} in $\Bbb{Z}\mathcal{S}$. The reader is referred to \cite[Chapter 3]{EGNO15} and \cite[Section 3]{Et02} for details.

For any $C\in\mathcal{S}$, let $\operatorname{FPdim}(C)$ be the maximal non-negative eigenvalue of the matrix of left multiplication by $C$. Since this matrix has non-negative entries, it follows from the Frobenius-Perron theorem that $\operatorname{FPdim}(C)$ exists. Furthermore, $\operatorname{FPdim}$ is the unique character of $\Bbb{Z}\mathcal{S}$ which takes non-negative values on $\mathcal{S}$.
\begin{lemma}\label{Prop:fpdim}
If $H_0$ is finite-dimensional, then for any $C\in\mathcal{S}$, we have $\operatorname{FPdim}(C)=\sqrt{\dim_{\k}(C)}$.
\end{lemma}
\begin{proof}
This is because $$C\mapsto \sqrt{\dim_{\k}(C)}$$
is exactly the unique character of $\Bbb{Z}\mathcal{S}$ which take non-negative values.
\end{proof}
\subsection{Properties for link quiver}
In this subsection, let $H$ be a non-cosemisimple Hopf algebra over $\k$ with the dual Chevalley property and $H_1/H_0$ is finite-dimensional. Note that according to Lemma \ref{Lem:P=H_1/H_0}, we know that $\mathcal{P}$ is a finite set in this situation.
Besides, for any matrix $\A=(a_{ij})_{m\times n}$ over $H$, denote the matrix
$\A^T:=(a_{ji})_{n\times m}$ and $S(\A):=(S(a_{ij}))_{m\times n}$, where $S$ is the antipode of $H$.

Now let us recall the concept of link quiver.
\begin{definition}\emph{(}\cite[Definition 4.1]{CHZ06}\emph{)}
Let $H$ be a coalgebra over $\k$. The link quiver $\mathrm{Q}(H)$ of $H$ is defined as follows: the vertices of $\mathrm{Q}(H)$ are the elements of $\mathcal{S}$; for any simple subcoalgebra $C, D\in \mathcal{S}$ with $\dim_{\k}(C)=r^2, \dim_{\k}(D)=s^2$, there are exactly $\frac{1}{rs}\dim_{\k}((C\wedge D)/(C+D))$ arrows from $D$ to $C$.
\end{definition}

With the notations in Section \ref{section3}, we can view $^{\C}{\mathcal{P}}^{\D}$ as the set of arrows from vertex $D$ to vertex $C$, view ${\mathcal{P}^{\D}}$ as the set of arrows with start vertex $D$ and view ${^{\C}\mathcal{P}}$ as the set of arrows with end vertex $C$. Similar statements can also be applied to $\mathcal{P}^\prime$.

Now we start to study the properties for the link quiver of a Hopf algebra with the dual Chevalley property.
\begin{lemma}\label{lemma:P^1=1^P}
Let $H$ be a non-cosemisimple Hopf algebra over $\k$ with the dual Chevalley property. Denote ${}^1\mathcal{S}=\{C\in\mathcal{S}\mid \k1+C\neq\k1\wedge C \}$, $\mathcal{S}^1=\{C\in\mathcal{S}\mid C+\k1\neq C\wedge \k1\}$. Then
\begin{itemize}
  \item[(1)]$\mid{^1\mathcal{P}}\mid\geq1$;
  \item[(2)]$\mid{^1\mathcal{P}}\mid=\mid\mathcal{P}^1\mid$;
  \item[(3)]$C\in{}^1\mathcal{S}$ if and only if $S(C)\in\mathcal{S}^1$.
  \end{itemize}
\end{lemma}

\begin{proof}
\begin{itemize}
  \item[(1)]At first, we try to find a non-trivial $(1, \F)$-primitive matrix for some $\F\in\mathcal{M}$. This can be obtained by the same reason in the proof of \cite[Lemma 4.7(1)]{LL22}, but here we prove it in another way. When $H\neq H_0$, it follows from Lemma \ref{Lem:P=H_1/H_0} that $\mathcal{P}\neq0$. So there exists some non-trivial $(\C, \D)$-primitive matrix $\X\in\mathcal{P}$. Let $K S(\C)^TK^{-1}\in\mathcal{M}$ be the basic multiplicative matrix of $S(C)$, where $K$ is some invertible matrix over $\k$. Since $S(C)\cdot C$ contains $\k1$ with a nonzero coefficient, by Lemma \ref{lemma:Kronecker}, we have
 \begin{eqnarray*}
     && (KS(\C)^TK^{-1})\odot^\prime
\left(\begin{array}{cc}
\C&\X\\
0&\D
 \end{array}\right)\\
 &=&(K\odot^\prime I) \left(S(\C)^T\odot^\prime
\left(\begin{array}{cc}
\C&\X\\
0&\D
 \end{array}\right)\right)
 (K^{-1}\odot^\prime I)\\
 &\sim&
 \left(\begin{array}{cccccccccc}
    \E_1 &  &  & &\X_{1,1} & \cdots & \X_{1,u}  \\
    & \ddots &  & &\vdots &  & \vdots  \\
     & & \E_t & &\X_{t,1} & \cdots & \X_{t,u} \\
     &&&1 &\X_{t+1,1} & \cdots & \X_{t+1,u}\\
     &  &  && F_1 &     \\
     & 0 &  &&  & \ddots    \\
     &  &  &  &&  & \F_u
  \end{array}\right),
\end{eqnarray*}
where $I$ is the identity matrix over $\k$, $\E_i, \F_j\in\mathcal{M}$ for any $1\leq i\leq t, 1\leq j\leq u.$
 By Lemma \ref{Lemma:CXno0}, there exists some $k$ such that $\X_{t+1, k}$ is non-trivial, where $1\leq k\leq u$. Thus we get a non-trivial $(1, \F_k)$-primitive matrix $\X_{t+1, k}$. \\
 According to Lemma \ref{Lemma:XinCPD}, we have  $$0\neq\span(\overline{\X_{t+1, k}})\subseteq({}^1H_1{}^{\F_k}+H_0)/H_0\subseteq({}^1H_1+H_0)/H_0.$$
 Consequently, it follows from Corollary \ref{coro:CP,PD} that $$\mid{^1\mathcal{P}}\mid\geq1.$$
  \item[(2)]
  For any $\C\in\mathcal{M}$, it is not difficult to verify that $S(\Y)^T$ is a non-trivial $(S(\C)^T, 1)$-primitive matrix, where $\Y\in{}^1\mathcal{P}$ is a non-trivial $(1, \C)$-primitive matrix.
   According to (\ref{cpd=sumx}) in the proof of Theorem \ref{coro:BXcomplete}, we have
     \begin{eqnarray*}
  S\big(({}^1H_1{}^C+H_0)/H_0\big)&=&S\big(\bigoplus_{\Y\in {}^1\mathcal{P}^{\C}} \span(\overline{\Y})\big)\\
&\subseteq& \sum_{\Y\in{}^1\mathcal{P}^{\C}}\span(\overline{S(\Y)^T}) \\
&\subseteq&({}^{S(C)}H_1{}^1+H_0)/H_0,
   \end{eqnarray*}
   the last inclusion is due to Lemma \ref{Lemma:XinCPD}.
It follows from \cite[Corollay 3.6]{Rad77} that $S$ is bijective. Hence we have
  $$\dim_{\k}(({}^1H_1{}^C+H_0)/H_0))\leq \dim_{\k}(({}^{S(C)}H_1{}^1+H_0)/H_0).$$ This implies that
  \begin{eqnarray*}
  \mid{^1\mathcal{P}^{\C}}\mid&=&\frac{1}{\sqrt{\dim_{\k}(C)}}\dim_{\k}(({}^1H_1{}^C+H_0)/H_0))\\
  &\leq&\frac{1}{\sqrt{\dim_{\k}(S(C))}}\dim_{\k}(({}^{S(C)}H_1{}^1+H_0)/H_0))\\
  &=&\mid{}^{KS(\C)^TK^{-1}}\mathcal{P}^1\mid,
  \end{eqnarray*}
  where $K$ is some invertible matrix over $\k$ such that $K S(\C)^TK^{-1}\in\mathcal{M}$ is a basic multiplicative matrix of $S(C)$.
  By Corollary \ref{coro:CP,PD} and the fact that $S$ is a permutation on $\mathcal{S}$, we have
   \begin{eqnarray*}
  \mid{^1\mathcal{P}}\mid&=& \sum_{\C\in\mathcal{M}} \mid{^1\mathcal{P}^{\C}}\mid \\
  &\leq&  \sum_{\C\in\mathcal{M}}\mid{}^{KS(\C)^TK^{-1}}\mathcal{P}^1\mid\\
  &=&\sum_{\C\in\mathcal{M}}\mid{}^{\C}\mathcal{P}^1\mid\\
  &=&\mid{\mathcal{P}^1}\mid.
  \end{eqnarray*}
  Next we adopt the same procedure to deal with $\mathcal{P}^{\prime 1}$, we get
  \begin{eqnarray*}
  \mid\mathcal{P}^{\prime 1}\mid\leq\mid{^1\mathcal{P}^\prime}\mid.
  \end{eqnarray*}
  According to Corollaries \ref{coro:CP,PD} and \ref{coro:CP^prime,PD},
  \begin{eqnarray*}
  \mid{^1\mathcal{P}}\mid=\sum_{\C\in\mathcal{M}} \mid{^1\mathcal{P}^{\C}}\mid
  =\sum_{\C\in\mathcal{M}} \mid{^1\mathcal{P}^{\prime\C}}\mid
  =\mid{^1\mathcal{P}^\prime}\mid.
  \end{eqnarray*}
  A similar arguments shows that
  \begin{eqnarray}\label{P1=Pprime1}
  \mid{\mathcal{P}^1}\mid=\mid{\mathcal{P}^{\prime 1}}\mid.
  \end{eqnarray}
   Thus the proof is completed.
    \item[(3)]It is straightforward to know that \begin{eqnarray*}
    C\in{}^1\mathcal{S}&\Longleftrightarrow& \k 1+C\neq \k 1\wedge C\\ &\Longleftrightarrow& S(C)+\k 1\neq S(C)\wedge \k1\\ &\Longleftrightarrow&  S(C)\in \mathcal{S}^1.
     \end{eqnarray*}
   \end{itemize}
\end{proof}
For convenience, denote $\mathcal{S}=\{C_i\mid i\in I\}$ be the set of all the simple subcoalgebras of $H$. For any $C_i, C_j\in\mathcal{S}$, let $C_i\cdot C_j=\sum\limits_{t\in I}\alpha_{i,j}^tC_t$, where $\alpha_{i,j}^t\in\Bbb{Z}_+$.
Moreover, we denote $\mathcal{M}=\{\C_j\mid i\in I\}$, such that each $\C_j\in\mathcal{M}$ is the basic multiplicative matrix of $C_j\in\mathcal{S}$.
By Remarks \ref{remark:ubc} and \ref{remark:ubcprime}, we can show the following results.
\begin{lemma}\label{lem:numberinPX}
Let $H$ be a non-cosemisimple Hopf algebra over $\k$ with the dual Chevalley property.
\begin{itemize}
  \item[(1)]For any $\Y\in{}^{1}\mathcal{P}$, where $\Y$ is a non-trivial $(1, \C_j)$-primitive matrix and $\C_j\in\mathcal{M}$, let $\beta_{ij}$ be the cardinal number of $^{\C_i}\mathcal{P}_{\Y}$. Then $\beta_{ij}=\sum\limits_{t\in I}\alpha_{i,j}^t\geq1$;
  \item[(2)]For any $\Y^\prime\in\mathcal{P}^{\prime 1}$, where $\Y^\prime$ is a non-trivial $(\C_j, 1)$-primitive matrix and $\C_j\in\mathcal{M}$, let $\beta_{ij}^\prime$ be the cardinal number of $\mathcal{P}_{\Y}^{\prime \C_i}$. Then $\beta_{ij}^\prime=\sum\limits_{t\in I}\alpha_{i,j}^t\geq1$.
  \end{itemize}
  \end{lemma}
For any $\Y\in{}^1\mathcal{P}$ and $\C_i\in\mathcal{M}$, denote $$\mathcal{P}^{\C_i}_{\Y}:=\mathcal{P}^{\C_i}\cap \mathcal{P}_{\Y}.$$
Using the lemma above, we can acquire further properties.
\begin{corollary}\label{coro:numberinPX}
Let $H$ be a non-cosemisimple Hopf algebra over $\k$ with the dual Chevalley property. Then for any non-trivial $(1, \C_j)$-primitive matrix  $\Y\in{}^1\mathcal{P}$, where $\C_j\in\mathcal{M}$, we have
\begin{itemize}
  \item[(1)]$\mid{{}^{\C_i}\mathcal{P}_{\Y}}\mid\geq1$, $\mid{\mathcal{P}^{\C_i}_{\Y}}\mid\geq1$ hold for all $\C_i\in\mathcal{M}$;
  \item[(2)]$\mid\mathcal{P}_{\Y}^1\mid=1$.
  \end{itemize}
\end{corollary}
\begin{proof}
\begin{itemize}
  \item[(1)]For any $\C_i\in\mathcal{M}$, it is apparent from Lemma \ref{lem:numberinPX} that $$\beta_{ij}=\mid{{}^{C_i}\mathcal{P}_{\Y}}\mid\geq 1.$$ Since $\Bbb{Z}\mathcal{S}$ is a unital based ring, according to Lemma \ref{lemma:transitive}, there exists some simple subcolagebra $C_t\in\mathcal{S}$ such that $C_t\cdot C_j$ contains $C_i$ with a nonzero coefficient. Now we consider
      $$
     \left(\begin{array}{cc}
I&0\\
0&L_{\C_t, \C_j}
 \end{array}\right)
      \left(\C_t\odot^\prime
      \left(\begin{array}{cc}
1&\Y\\
0&\C_j
 \end{array}\right)\right)
  \left(\begin{array}{cc}
I&0\\
0&L^{-1}_{\C_t, \C_j}
 \end{array}\right),
      $$
 where $I$ is the identity matrix over $\k$ and $L_{\C_t, \C_j}$ is an invertible matrix over $\k$ which is defined in Section \ref{section3}. It follows from Lemma \ref{Lemma:CXno0} that there exists some non-trivial $(\C_t, \C_i)$-primitive matrix $\Z\in {}^{\C_t}\mathcal{P}_{\Y}\subseteq \mathcal{P}_{\Y}$, where $\C_t\in\mathcal{M}$. On the other hand, we know that $\Z\in {}^{\C_t}\mathcal{P}^{\C_i}\subseteq \mathcal{P}^{\C_i},$ where the last inclusion is due to Corollary \ref{coro:CP,PD}. It follows that $$\Z\in \mathcal{P}^{\C_i}\cap \mathcal{P}_{\Y}.$$Thus $$\mid{\mathcal{P}^{\C_i}_{\Y}}\mid\geq1.$$
\item[(2)]Choosing $C_i=\k 1$ in $(1)$, we know that
\begin{eqnarray*}
\mid{\mathcal{P}^{1}_{\Y}}\mid&\geq&1=\mid{}^1\mathcal{P}_{\Y}\mid,
\end{eqnarray*}
where $${}^1\mathcal{P}_{\Y}=\{\Y\}.$$
Since $$\mathcal{P}^{1}=\mathcal{P}^{1}\cap \mathcal{P}=\bigcup_{\Z\in{}^1\mathcal{P}}(\mathcal{P}^{1}\cap \mathcal{P}_{\Z}),$$it follows from Corollary \ref{coro:P_X} that
\begin{eqnarray*}
    \mid{\mathcal{P}^{1}}\mid= \sum_{\Z\in{}^1\mathcal{P}}\mid {\mathcal{P}_{\Z}^{1}} \mid
   \geq \sum_{\Z\in{}^1\mathcal{P}}\mid{}^1\mathcal{P}_{\Z}\mid
    =\mid{{}^{1}\mathcal{P}}\mid.
    \end{eqnarray*}
    But by Lemma \ref{lemma:P^1=1^P}, we have
    $$\mid\mathcal{P}^1\mid=\mid{}^1\mathcal{P}\mid, $$
    which follows that the cardinal number $\mid{\mathcal{P}^{1}_{\Z}}\mid$ can only equal to $1$ for each $\Z\in{}^1\mathcal{P}$.
  \end{itemize}
  \end{proof}
For any $C, D\in\mathcal{S}$ with basic multiplicative matrices $\C, \D\in\mathcal{M}$ respectively.
Recall that $C$, $D$ are said to be \textit{directly linked} in $H$ if $C+D$ is a proper subspace of $C\wedge D+D\wedge C$.
Note that by \cite[Lemma 3.6 (2)]{Li22a} that $C$, $D$ are directly linked in $H$ if and if only there exists some $(\C, \D)$-primitive or $(\D, \C)$-primitive matrix, which is non-trivial.
We can use this notion to describe ${}^1\mathcal{S}$ and $\mathcal{S}^1$.

The following proposition exactly recovers the definition of Hopf quiver.
\begin{proposition}\label{proposition:link1dim===}
Let $H$ be a non-cosemisimple Hopf algebra over $\k$ with the dual Chevalley property. If all the simple subcoalgebras directly linked to $\k1$ are $1$-dimensional, then we have
$\mid{^{\C}\mathcal{P}}\mid=\mid{\mathcal{P}^{\C}}\mid=\mid{^1\mathcal{P}}\mid$,  for any $\C\in \mathcal{M}$.
\end{proposition}

\begin{proof}
Suppose that $y\in{}^1\mathcal{P}$ is a non-trivial $(1, g)$-primitive matrix of size 1 for some $g\in G(H)$. For any $C\in\mathcal{S}$, it is straightforward to show that $C g$ is also a simple subcoalgebra of $H$. By (\ref{^BPY}), we know that for any $\C\in\mathcal{M}$,
$${}^{\C}\mathcal{P}_\Y=\{\C y\}\;\;\;\;\text{and}\;\;\;\; {}^{1}\mathcal{P}_\Y=\{y\}.$$
This means that $$\mid{}^{\C}\mathcal{P}_\Y\mid=1=\mid{}^1\mathcal{P}_\Y\mid.$$
Therefore, we have
\begin{eqnarray}\label{4.1}
\mid{^{\C}\mathcal{P}}\mid=\sum_{\Y\in{}^1\mathcal{P}}\mid{}^{\C}\mathcal{P}_\Y\mid
=\sum_{\Y\in{}^1\mathcal{P}}\mid{}^1\mathcal{P}_\Y\mid
=\mid{^1\mathcal{P}}\mid.
\end{eqnarray}
By the same method as employed above, we can show that \begin{eqnarray}\label{4.2}
\mid{\mathcal{P}^{\prime \C}}\mid=\mid{\mathcal{P}^{\prime 1}}\mid.
\end{eqnarray}
According to Corollaries \ref{coro:CP,PD} and \ref{coro:CP^prime,PD}, for any $\D\in\mathcal{M}$, we have
\begin{eqnarray}\label{4.3}
\mid{\mathcal{P}^{\D}}\mid=\sum_{\C\in\mathcal{M}}\mid{^{\C}\mathcal{P}^{\D}}\mid
=\sum_{\C\in\mathcal{M}}\mid{^{\C}\mathcal{P}^{\prime \D}}\mid
=\mid{\mathcal{P}^{\prime\D}}\mid.
\end{eqnarray}
It follows from Lemma \ref{lemma:P^1=1^P} that
$$\mid{^{\C}\mathcal{P}}\mid\overset{(\ref{4.1})}{=}\mid{^1\mathcal{P}}\mid=\mid{\mathcal{P}^1}\mid\overset{(\ref{4.3})}{=}\mid{\mathcal{P}^{\prime 1}}\mid\overset{(\ref{4.2})}{=}\mid{\mathcal{P}^{\prime \C}}\mid\overset{(\ref{4.3})}{=}\mid{\mathcal{P}^{ \C}}\mid.$$
\end{proof}
Now let us focus on a special situation.
\begin{corollary}\label{lemma:Pc=cP=1}
Let $H$ be a non-cosemisimple Hopf algebra over $\k$ with the dual Chevalley property. Then the followings are equivalent:
\begin{itemize}
  \item[(1)]$\mid{}^{\C}\mathcal{P}\mid=\mid\mathcal{P}^{\C}\mid=1 $ holds for all $C\in\mathcal{M}$;
  \item[(2)]$\mid{}^1\mathcal{P}\mid=1$ and the unique subcoalgebra $C\in{}^1\mathcal{S}$ is $1$-dimensional.
  \end{itemize}
\end{corollary}

\begin{proof}
Indeed, it follows from Proposition \ref{proposition:link1dim===} that (2) implies (1).
Conversely, assume that there exists some simple subcoalgebra $C\in\mathcal{S}_1$ such that $$\dim_{\k}(C)>1.$$  Suppose $K S(\C)^TK^{-1}\in\mathcal{M}$ is a basic multiplicative matrix of $S(C)$, where $K$ is some invertible matrix over $\k$. Since there is only one $1$ in the summand of $S(C)\cdot C$ in $\Bbb{Z}\mathcal{S}$, it follows from Lemma \ref{lem:numberinPX} that $$\mid{}^{KS(C)K^{-1}}\mathcal{P}\mid>1,$$
which is a contradiction to the assumption in (1).
\end{proof}
According to Proposition \ref{proposition:link1dim===}, we know that if all the simple subcoalgebras directly linked to $\k1$ are $1$-dimensional, we have
$$\mid{}^{1}\mathcal{P}\mid \big| \mid{}^{\C}\mathcal{P}\mid$$ and
$$\mid{}^{\C}\mathcal{P}\mid=\mid\mathcal{P}{}^{\C}\mid,$$ for any $\C\in\mathcal{M}.$
This motivates us forming the following conjecture, which we hope can guide future research about Hopf algebras with the dual Chevalley property.
\begin{conjecture}\label{4.11}Let H be a non-cosemisimple Hopf algebra over k with the dual Chevalley
property. Then we have
\begin{itemize}
\item[(1)]$\mid{}^{1}\mathcal{P}\mid \big| \mid{}^{\C}\mathcal{P}\mid$, for any $\C\in\mathcal{M}$;
\item[(2)]$\mid{}^{\C}\mathcal{P}\mid=\mid\mathcal{P}{}^{\C}\mid$, for any $\C\in\mathcal{M}$.
\end{itemize}
\end{conjecture}
According to \cite[Propositions 3.3.6(2) and 3.3.11]{EGNO15}, there exists a unique, up to scaling, nonzero element $R\in\Bbb{Z}\mathcal{S}\otimes_{\Bbb{Z}}\Bbb{C}$ such that $X\cdot R=\operatorname{FPdim}(X)R$ for all $X\in \Bbb{Z}\mathcal{S}$, and it satisfies the equality $R\cdot Y=\operatorname{FPdim}(Y)R$ for all $Y\in\Bbb{Z}\mathcal{S}$. Such an element $R$ is called a regular element of $\Bbb{Z}\mathcal{S}$. It is straightforward to show that the element $R=\sum\limits_{Y\in I}\operatorname{FPdim}(Y)Y$ is a regular element.
We can obtain a useful equation, which is needed in the next section.
\begin{lemma}\label{lem:fpequation}
 With the notations in Lemma \ref{lem:numberinPX}, suppose that $H_0$ is finite-dimensional. Then we have the following equation
 $$\sqrt{\dim_{\k}(C_k)}\left(\sum\limits_{i\in I} \sqrt{\dim_{\k}(C_i)}\right)=\sum\limits_{i\in I} \sqrt{\dim_{\k}(C_i)}\beta_{ik}.$$
\end{lemma}

\begin{proof}
By Lemma \ref{Prop:fpdim} and Lemma \ref{lem:numberinPX}, we know that  $$\beta_{ik}=\sum\limits_{t\in I}\alpha_{i,k}^t$$ and $$\operatorname{FPdim}(C_i)=\sqrt{\dim_{\k}(C_i)}$$  hold for any $i\in I$. It follows from \cite[Propositions 3.3.6(2) and 3.3.11]{EGNO15} that $$R=\sum\limits_{i\in I}\sqrt{\dim_{\k}(C_i)}C_i$$ is a regular element
and $$\sqrt{\dim_{\k}(C_k)}R=R\cdot C_k.$$
This means that
\begin{eqnarray*}
 \sqrt{\dim_{\k}(C_k)}\left(\sum\limits_{i\in I}\sqrt{\dim_{\k}(C_i)}C_i\right)&=&\left(\sum\limits_{i\in I}\sqrt{\dim_{\k}(C_i)}C_i\right)\cdot C_k\\
&=& \sum\limits_{i\in I}\sum\limits_{t\in I}\sqrt{\dim_{\k}(C_i)}\alpha_{i,k}^t C_t,
\end{eqnarray*}
 which follows that
 $$\sqrt{\dim_{\k}(C_k)}\sqrt{\dim_{\k}(C_t)}=\sum\limits_{i\in I}\sqrt{\dim_{\k}(C_i)}\alpha_{i,k}^t.$$
Thus we have $$\sqrt{\dim_{\k}(C_k)}\left(\sum\limits_{i\in I} \sqrt{\dim_{\k}(C_i)}\right)=\sum\limits_{i\in I} \sqrt{\dim_{\k}(C_i)}\beta_{ik}.$$
\end{proof}

At the end of this subsection, we recall the concept of link-indecomposable components of coalgebra $H$.
\begin{definition}\emph{(}\cite[Definition 1.1]{Mon95}\emph{)}
A subcoalgebra $H^\prime$ of coalgebra $H$ is called \textit{link-indecomposable} if the link quiver $\mathcal{Q}(H^\prime)$ of $H^\prime$ is connected (as an undirected graph).
A \textit{link-indecomposable component} of $H$ is a maximal link-indecomposable subcoalgebra.
\end{definition}
As a consequence, we obtain the following proposition.
\begin{proposition}\label{lemma:link1dim}
Let $H$ be a non-cosemisimple Hopf algebra over $\k$ with the dual Chevalley property. If all the simple subcoalgebras directly linked to $\k1$ are $1$-dimensional, then the link-indecomposable component $H_{(1)}$ containing $\k 1$ is a pointed Hopf algebra.
\end{proposition}
\begin{proof}
Suppose there exists a simple subcoalgebra $B$ with $\dim_{\k}(B)> 1$ such that some $1$-dimensional simple subcoalgebra $\k g$ contained in $H_{(1)}$ is directly linked to $B$, where $g\in G(H)$. Without the loss of generality, we can assume that there exists some non-trivial $(g, \B)$-primitive matrix for some $\B\in\mathcal{M}$. Now we consider
$$
 g^{-1}\odot^\prime
      \left(\begin{array}{cc}
g&\X\\
0&\B
 \end{array}\right)
 .$$
  We can get a $(1, g^{-1}\B)$-primitive matrix $g^{-1}\X$, where $g^{-1}\B$ is a multiplicative matrix of $g^{-1}B\in\mathcal{S}$. We know that $\k1$ is directly linked to the simple subcoalgebra $g^{-1}B$, which is a contradiction.
Therefore, it is directly from \cite[Theorem 4.8 (3)]{Li22a} that $H_{(1)}$ is a pointed Hopf algebra.
\end{proof}
\section{Corepresentation type}\label{section5}
In this section, we want to classify finite-dimensional Hopf algebras with the dual Chevalley properties of finite corepresentation type. The reader is referred to \cite{ARS95} and \cite{ASS06} for general background knowledge of representation theory.

Recall that a finite-dimensional algebra $A$ is said to be of \textit{finite representation type} if there are finitely many isoclasses of indecomposable $A$-modules. An algebra $A$ is said to be of \textit{infinite representation type}, if $A$ is not of finite representation type. A finite-dimensional coalgebra $C$ is said to be of \textit{finite corepresentation type}, if the dual algebra $C^*$ is of finite representation type. A finite-dimensional coalgebra $C$ is defined to be of \textit{infinite corepresentation type},
if $C^*$ is of infinite representation type.

Let $A$ (resp. $C$) be an algebra (resp. coalgebra) over $\k$ and $\{M_i\}_{ i\in I}$ be the complete set of isoclasses of simple left $A$-modules (resp. right $C$-comodules). The \textit{Ext quiver} $\Gamma(A)$ (resp. $\Gamma(C)$) of $A$ (resp. $C$) is an oriented graph with vertices indexed by $I$, and there are $\dim_{\k}\Ext^1(M_i, M_j)$ arrows from $i$ to $j$ for any $i, j\in I$.
To avoid confusion, for any Hopf algebra $H$ over $\k$, we denote the algebra's version of Ext quiver of $H$ by $\Gamma(H)^{\mathrm{a}}$ and denote the coalgebra's version of Ext quiver of $H$ by $\Gamma(H)^\mathrm{c}$.

\subsection{Corepresentation type and link quiver}
At first, let us recall the definition of separated quiver.
\begin{definition}\emph{(}cf. \cite[\textsection X. 2]{ARS95}\emph{)}
Let $A$ be a finite-dimensional algebra over $\k$ and $\Gamma(A)=(\Gamma_0, \Gamma_1)$ be its Ext quiver, where $\Gamma_0=\{1, 2, \cdots, n\}$. The separated quiver $\Gamma(A)_{s}$ of $A$ has $2n$ vertices $\{1, 2, \cdots, n, 1^\prime, 2^\prime, \cdots, n^\prime\}$ and an arrow $i \rightarrow j^\prime$ for every arrow $i \rightarrow j$ of $\Gamma(A)$.
\end{definition}

Let $H$ be a finite-dimensional Hopf algebra over $\k$ with the dual Chevalley property. For simplicity, we still use the notations in Sections \ref{section3} and \ref{section4}.
As mentioned in Section \ref{section4}, for any simple subcoalgebra $C, D\in\mathcal{S}$ with $\dim_{\k}(C)=r^2, \dim_{\k}(D)=s^2$, there are exactly $\frac{1}{rs}\dim_{\k}((C\wedge D)/(C+D))$ arrows from $D$ to $C$ in the link quiver $\mathrm{Q}(H)$ of $H$. Moreover, we have $$\mid{}^{\C}\mathcal{P}^{\D}\mid=\mid{}^{\C}\mathcal{P}^{\prime D}\mid=\frac{1}{rs}\dim_{\k}((C\wedge D)/(C+D)),$$
and there are exactly $\mid{}^{\C}\mathcal{P}\mid(=\mid{}^{\C}\mathcal{P}^\prime\mid)$ arrows with end vertex $C$ and $\mid\mathcal{P}^{\D}\mid(=\mid\mathcal{P}^{\prime D}\mid)$ arrows with start vertex $D$ in the link quiver $\mathrm{Q}(H)$ of $H$.

In order to solve the classification problems, we divide it into several different situations. Let us consider the first case.
\begin{proposition}\label{prop:2arrow=infin}
Let $H$ be a finite-dimensional non-cosemisimple Hopf algebra over $\k$ with the dual Chevalley property. If $\mid{}^1\mathcal{P}\mid\geq2$, then $H$ is of infinite corepresentative type.
\end{proposition}
\begin{proof}
We know that the $\k$-linear abelian category of finite-dimensional comodules over $H$ is isomorphic to the category of finite-dimensional modules over $H^*$. This means that the coalgebra's version of Ext quiver $\Gamma(H)^{\mathrm{c}}$ of $H$ is the same as the algebra's version of Ext quiver $\Gamma(H^*)^{\mathrm{a}}$ of $H^*$. According to \cite[Theorem 2.1 and Corollary 4.4]{CHZ06}, the link quiver $\mathrm{Q}(H)$ of $H$ coincides with the algebra's version of Ext quiver $\Gamma(H^*)^{\mathrm{a}}$ of $H^*$. \\Note that $H^*$ is Morita equivalent to a basic algebra $\mathcal{B}(H^*)$.
It suffices to prove that the basic algebra $\mathcal{B}(H^*)$ of $H^*$ is of infinite representative type. Let $J$ be the ideal generated by all the arrows in $\mathrm{Q}(H)$. By the Gabriel's theorem, there exists an admissible ideal $I$ such that $$\k\mathrm{Q}(H)/I\cong \mathcal{B}(H^*),$$ where $J^t\subseteq I \subseteq J^2$ for some integer $t\geq2$. Thus there exists an algebra epimorphism $$f:\mathcal{B}(H^*)\rightarrow \k\mathrm{Q}(H)/J^2.$$ It is enough to show that $\k\mathcal{Q}(H)/J^2$ is of infinite representative type. Since the Jacobson radical of $\k\mathcal{Q}(H)/J^2$ is $J/J^2$, we know that $\k\mathrm{Q}(H)/J^2$ is an artinian algebra with radical square zero. \\
Now assume on the countrary that $\k\mathrm{Q}(H)/J^2$ is of finite representation type. It follows from \cite[X.2 Theorem 2.6]{ARS95} that the separated quiver $\mathrm{Q}(H)_s$ of $\k\mathrm{Q}(H)/J^2$ is a finite disjoint union of Dynkin diagrams.\\
Since $\mid{}^1\mathcal{P}\mid\geq2$, it follows from Corollary \ref{coro:P_X} that
\begin{eqnarray}\label{5.1}
\mid{}^{C}\mathcal{P}\mid=\sum_{\Y\in{}^1\mathcal{P}}\mid{}^{C}\mathcal{P}_{\Y}\mid\geq\mid{}^1\mathcal{P}\mid\geq2.
 \end{eqnarray}
 According to Corollary \ref{coro:numberinPX}, we have \begin{eqnarray}\label{5.2}
 \mid\mathcal{P}^{C}\mid=\sum_{\Y\in{}^1\mathcal{P}}\mid\mathcal{P}^{C}_{\Y}\mid\geq\mid{}^1\mathcal{P}\mid\geq2
 \end{eqnarray} for all $\C\in\mathcal{M}$.\\
By a discussion on $\mid{}^1\mathcal{S}\mid$, we aim to find a contradiction to $\mathrm{Q}(H)_s$ being a finite disjoint union of Dynkin diagrams.
\begin{itemize}
\item[i)]
When $\mid{}^1\mathcal{S}\mid=1,$ then the separated quiver $\mathrm{Q}(H)_s$ must contain
$$
\begin{tikzpicture}
\filldraw [black] (1,0) circle (2pt) node[anchor=west]{$\k 1^\prime$};
\filldraw [black] (-1,0) circle (2pt)node[anchor=north]{};
\draw[thick, ->] (-0.9,0.1).. controls (-0.1,0.1) and (0.1,0.1) .. node[anchor=south]{}(0.9,0.1) ;
\draw[thick, ->] (-0.9,-0.1).. controls (-0.1,-0.1) and (0.1,-0.1) .. node[anchor=south]{}(0.9,-0.1);
\end{tikzpicture}
$$
as a sub-quiver. The above quiver is a Kronecker quiver whose underlying graph is not a Dynkin diagram. We know that $\k\mathrm{Q}(H)/J^2$ is of infinite representation.\\
\item[ii)]When $\mid{}^1\mathcal{S}\mid\geq 2,$ by (\ref{5.1}) and (\ref{5.2}),
the separated quiver $\mathrm{Q}(H)_s$ contains a sub-quiver of the form
$$
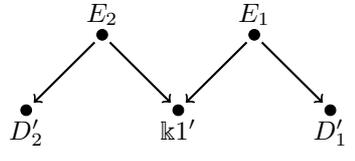
\begin{tikzpicture}
\filldraw [black] (1,0) circle (2pt) node[anchor=south]{$E_1 $};
\filldraw [black] (0,-1) circle (2pt) node[anchor=north]{$\k1^\prime$};
\filldraw [black] (-1,0) circle (2pt)node[anchor=south]{$E_2 $};
\filldraw [black] (2,-1) circle (2pt) node[anchor=north]{$D_1^\prime$};
\filldraw [black] (-2,-1) circle (2pt)node[anchor=north]{$D_2^\prime$};
\draw[thick, ->] (-0.9,-0.1).. controls (-0.9,-0.1) and (-0.9,-0.1) .. node[anchor=north]{}(-0.1,-0.9) ;
\draw[thick, ->] (0.9,-0.1).. controls (0.9,-0.1) and (0.9,-0.1) .. node[anchor=south]{}(0.1,-0.9);
\draw[thick, ->] (-1.1,-0.1).. controls (-1.1,-0.1) and (-1.1,-0.1) .. node[anchor=north]{}(-1.9,-0.9) ;
\draw[thick, ->] (1.1,-0.1).. controls (1.1,-0.1) and (1.1,-0.1) .. node[anchor=south]{}(1.9,-0.9);
\end{tikzpicture},
$$
where $E_1\neq E_2$ and there are 3 possible situation.\\
If $D_1^\prime=\k1^\prime$ or $D_2^\prime=\k1^\prime$, the separated quiver of $\mathrm{Q}(H)$ contains a Kronecker quiver as a sub-quiver.\\
If $D_1^\prime=D_2^\prime\neq \k1^\prime$, the separated quiver of $\mathrm{Q}(H)$ contains a sub-quiver whose underlying graph is $\tilde{A_{n}}$ for some $n\geq3$ and it is a Euclidean diagram. Therefore, $\k\mathrm{Q}(H)/J^2$ is of infinite representation type.\\
 If $D_1^\prime, D_2^\prime, \k1^\prime$ are distinct from each other, the separated quiver $\mathrm{Q}(H)_s$ contains the following sub-quiver
$$
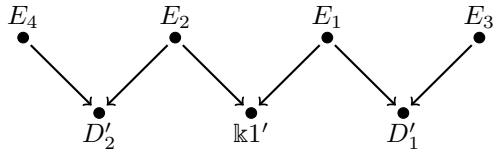
\begin{tikzpicture}
\filldraw [black] (1,0) circle (2pt) node[anchor=south]{$E_1 $};
\filldraw [black] (0,-1) circle (2pt) node[anchor=north]{$\k1^\prime$};
\filldraw [black] (-1,0) circle (2pt)node[anchor=south]{$E_2 $};
\filldraw [black] (2,-1) circle (2pt) node[anchor=north]{$D_1^\prime$};
\filldraw [black] (-2,-1) circle (2pt)node[anchor=north]{$D_2^\prime$};
\filldraw [black] (3,0) circle (2pt) node[anchor=south]{$E_3$};
\filldraw [black] (-3,0) circle (2pt)node[anchor=south]{$E_4$};
\draw[thick, ->] (-0.9,-0.1).. controls (-0.9,-0.1) and (-0.9,-0.1) .. node[anchor=north]{}(-0.1,-0.9) ;
\draw[thick, ->] (0.9,-0.1).. controls (0.9,-0.1) and (0.9,-0.1) .. node[anchor=south]{}(0.1,-0.9);
\draw[thick, ->] (-1.1,-0.1).. controls (-1.1,-0.1) and (-1.1,-0.1) .. node[anchor=north]{}(-1.9,-0.9) ;
\draw[thick, ->] (1.1,-0.1).. controls (1.1,-0.1) and (1.1,-0.1) .. node[anchor=south]{}(1.9,-0.9);
\draw[thick, ->] (-2.9,-0.1).. controls (-2.9,-0.1) and (-2.9,-0.1) .. node[anchor=north]{}(-2.1,-0.9) ;
\draw[thick, ->] (2.9,-0.1).. controls (2.9,-0.1) and (2.9,-0.1) .. node[anchor=south]{}(2.1,-0.9);
\end{tikzpicture}.
$$
Then if $E_4=E_i$ for some $i=1, 2, 3$, it is evident that $\k\mathrm{Q}(H)/J^2$ is of infinite representation type.
Otherwise, we repeat above process. Since $\mathrm{S}$ is a finite set, the separated quiver $\mathrm{Q}(H)_s$ either contains the Kronecker quiver as a sub-quiver or contains a sub-quiver whose underlying graph is $\tilde{A_{n}}$ for some $n\geq3$.
 \end{itemize}
As a conclusion, $\k\mathrm{Q}(H)/J^2$ is of infinite representation type, and this implies that $H$ is of infinite corepresentative type.
\end{proof}
Recall that an algebra is said to be \textit{Nakayama}, if each indecomposable projective left and right module has a unique composition series. It is well-known that a basic algebra $A$ is Nakayama if and only if every vertex of the Ext quiver of $A$ is the start vertex of at most one arrow and the end vertex of at most one arrow
(see \cite[\textsection V. 2. Theorem 2.6]{ASS06}).

Next we consider the second case.
\begin{proposition}\label{Prop:finite corep}
Let $H$ be a finite-dimensional non-cosemisimple Hopf algebra over $\k$ with the dual Chevalley property. If $\mid{}^1\mathcal{P}\mid=1$ and the unique subcoalgebra $C\in{}^1\mathcal{S}$ is $1$-dimensional, then $H$ is of finite corepresentative type.
\end{proposition}
\begin{proof}
From the proof of Proposition \ref{prop:2arrow=infin}, we know that the link quiver $\mathrm{Q}(H)$ of $H$ is the same as the Ext quiver $\Gamma(H^*)^\mathrm{a}$ of $H^*$.
Using Lemma \ref{lemma:Pc=cP=1}, we can find
 $$\mid{}^{\C}\mathcal{P}\mid=\mid\mathcal{P}{}^{\C}\mid=1,\;\;\;\; \C\in\mathcal{M},$$which means that the basic algebra $\mathcal{B}(H^*)$ is a Nakayama algebra. It follows from \cite[\textsection VI. Theorem 2.1]{ARS95} that the Nakayama algebra $\mathcal{B}(H^*)$ is of finite representation type, which implies that $H$ is of finite corepresentation type.
\end{proof}
Note that since $H$ is finite-dimensional, $\Bbb{Z}\mathcal{S}$ is a fusion ring with $\Bbb{Z}_+$-basis $\mathcal{S}=\{C_i\}_{i\in I}$. Suppose that $C_i\cdot C_j= \sum\limits_{t\in I}\alpha_{ij}^tC_t,$ for any $C_i, C_j\in\mathcal{S}$. By the proof of Proposition \ref{Prop:basedring}, the involution of $I$ is decided by $S$,  that is $C_{i^*}=S(C_i)$.

Before proceeding further, let us give the following lemma.
\begin{lemma}\label{lemma:numberinquiver}
Let $H$ be a finite-dimensional non-cosemisimple Hopf algebra over $\k$ with the dual Chevalley property. Let $\mid{}^1\mathcal{P}\mid=1$ and $C_k$ be the unique subcoalgebra contained in ${}^1\mathcal{S}$.
\begin{itemize}
\item[(1)]The number of arrows with end vertex $C_i$ in $\mathrm{Q}(H)$ is equal to $\sum\limits_{t\in I}\alpha_{ik}^t$, and the number of arrows with start vertex $C_i$ in $\mathrm{Q}(H)$ is equal to $\sum\limits_{t\in I}\alpha_{ik^*}^t$;
\item[(2)]The number of arrows from $C_t$ to $C_i$ in $\mathrm{Q}(H)$ is equal to $\alpha_{ik}^t$ and we have $\alpha_{ik}^t=\alpha_{tk^*}^i.$
\end{itemize}
\end{lemma}
\begin{proof}
\begin{itemize}
\item[(1)]According to Lemma \ref{lemma:P^1=1^P} (2) and (\ref{P1=Pprime1}) in its proof, we know that
 $$\mid\mathcal{P}{}^{\prime 1}\mid=\mid\mathcal{P}{}^1\mid=\mid{}^1\mathcal{P}\mid=1.$$ Suppose $${}^1\mathcal{P}=\{\Y\}$$ and $$\mathcal{P}{}^{\prime 1}=\{\Y^\prime\}.$$ Combining (\ref{definition:^BP}) and Lemma \ref{lem:numberinPX}, we have $$\mid{}^{\C_i}\mathcal{P}\mid=\mid{}^{\C_i}\mathcal{P}_{\Y}\mid=\sum\limits_{t\in I}\alpha_{ik}^t.$$
 This means that the number of arrows with end vertex $C_i$ in $\mathrm{Q}(H)$ is equal to $\sum\limits_{t\in I}\alpha_{ik}^t$.
 A similar argument shows that the number of arrows with start vertex $C_i$ in $\mathrm{Q}(H)$ is equal to $\sum\limits_{t\in I}\alpha_{ik^*}^t$.
\item[(2)]In $\Bbb{Z}\mathcal{S}$, we have $$S(C_k)\cdot S(C_i)=\sum\limits_{t\in I}\alpha_{ik}^tS(C_t)=\sum\limits_{t\in I}\alpha_{k^*i^*}^{t^*}S(C_t).$$
It follows from \cite[Proposition 3.1.6]{EGNO15} that $$\alpha_{ik}^t=\alpha_{k^*i^*}^{t^*}=\alpha_{tk^*}^i.$$
Moreover, by (\ref{definition:^BP}), we can find that $$^{C_i}\mathcal{P}={}^{C_i}\mathcal{P}_{\Y}.$$
It follows from Theorem \ref{coro:BXcomplete} that $$\mid{}^{C_i}\mathcal{P}^{\C_t}\mid=\alpha_{ik}^t.$$
Thus the number of arrows from $C_t$ to $C_i$ in $\mathrm{Q}(H)$ is equal to $\alpha_{ik}^t$.
\end{itemize}
\end{proof}
Using Lemma \ref{lemma:numberinquiver}, now we can turn to the last situation.
\begin{proposition}\label{dimC4,9}
Let $H$ be a finite-dimensional non-cosemisimple Hopf algebra over $\k$ with the dual Chevalley property. Let $\mid{}^1\mathcal{P}\mid=1$ and $C_k$ be the unique subcoalgebra contained in ${}^1\mathcal{S}$. If $\dim_{\k}(C_k)\geq4$, then $H$ is of infinite corepresentative type.
\end{proposition}
\begin{proof}
Proceeding as in the proof of Proposition \ref{prop:2arrow=infin}, we need only to observe the separated quiver of $\k\mathrm{Q}(H)/J^2$.\\
If $\dim_{\k}(C_k)\geq9$, since $\beta_{1k}=1$, it follows from Lemma \ref{lem:fpequation} that there exists at least one subcoalgebra $C_u$ such that $$\beta_{uk}=\sum\limits_{t\in I}\alpha_{uk}^t\geq4.$$ The separated quiver of $\mathrm{Q}(H)$ contains a vertex which is the end vertex of at least $4$ arrows.
Evidently, the underlying graph of this separated quiver is not the union of Dynkin diagrams, thus $H$ is of infinite corepresentative type.\\
If $\dim_{\k}(C_k)=4$, we deal with this situation through classified discussion. In the following part, for any $n\geq 2$, let $\mathcal{S}(n)$ be the set of all the $n^2$-dimensional simple subcoalgebras of $H$ and let $G(H)$ be the set of all the group-like elements.
According to Lemma \ref{lemma:numberinquiver} (2), we know that the following two numbers are equal:
\begin{itemize}
\item[-]The number of $C_t$ contained in $C_i\cdot C_k$;
\item[-]The number of $C_i$ contained in $C_t\cdot S(C_k)$.
\end{itemize}
Now let us start discussing different situations.
\begin{itemize}
  \item[(I)]Suppose that $$S(C_k)\cdot C_k=\k1+\k g_1+\k g_2+\k g_3$$ in $\Bbb{Z}\mathcal{S}$, where $g_1, g_2, g_3\in G(H).$ According to Lemma \ref{lemma:numberinquiver} (1), the separated quiver $\mathrm{Q}(H)_s$ contains a vertex which is the end vertex of $4$ arrows and it can not be a finite disjoint union of Dynkin diagrams.
We know that $H$ is of infinite corepresentation type.\\
 Note that if there exists some vertex in $\mathrm{Q}(H)_s$ which is the end vertex or the start vertex of at least $4$ arrows, then a similar arguments shows that $H$ is of infinite corepresentation type. For simplicity, in the following proof, we will no longer consider the occurrence of this situation.
\item[(II)] Suppose that in $\Bbb{Z}\mathcal{S}$, we have $$S(C_k)\cdot C_k=\k1+\k g_1+D_1^{(2)},$$ for some $g_1\in G(H)$ and $D_1^{(2)}\in\mathcal{S}(2)$.
\begin{itemize}
  \item[i)]If $$D_1^{(2)}\cdot S(C_k)=S(C_k)+\k g_2+\k g_3,$$ where $g_2, g_3\in G(H)$. Using Lemma \ref{lemma:numberinquiver} (2), the separated quiver of $\k\mathrm{Q}(H)/J^2$ either contains a sub-quiver of the form
$$
\begin{tikzpicture}
\filldraw [black] (1,0) circle (2pt) node[anchor=south]{};
\filldraw [black] (-1,0) circle (2pt)node[anchor=north]{};
\draw[thick, ->] (-0.9,0.1).. controls (-0.1,0.1) and (0.1,0.1) .. node[anchor=south]{}(0.9,0.1) ;
\draw[thick, ->] (-0.9,-0.1).. controls (-0.1,-0.1) and (0.1,-0.1) .. node[anchor=south]{}(0.9,-0.1);
\end{tikzpicture},
$$
or contains
$$
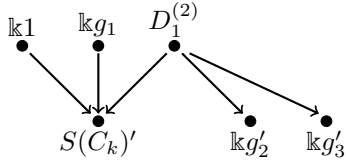
\begin{tikzpicture}
\filldraw [black] (0,0) circle (2pt) node[anchor=south]{$\k1$};
\filldraw [black] (1,0) circle (2pt)node[anchor=south]{$\k g_1$};
\filldraw [black] (2,0) circle (2pt)node[anchor=south]{$D_1^{(2)}$};
\filldraw [black] (1,-1) circle (2pt) node[anchor=north]{$S(C_k)^\prime$};
\filldraw [black] (3,-1) circle (2pt)node[anchor=north]{$\k g_2^\prime$};
\filldraw [black] (4,-1) circle (2pt)node[anchor=north]{$\k g_3^\prime$};
\draw[thick, ->] (0.1,-0.1).. controls (0.1,-0.1) and (0.1,-0.1) .. node[anchor=south]{}(0.9,-0.9) ;
\draw[thick, ->] (1,-0.1).. controls (1,-0.1) and (1,-0.1) .. node[anchor=south]{}(1,-0.9) ;
\draw[thick, ->] (1.9,-0.1).. controls (1.9,-0.1) and (1.9,-0.1) .. node[anchor=south]{}(1.1,-0.9) ;
\draw[thick, ->] (2.1,-0.1).. controls (2.1,-0.1) and (2.1,-0.1) .. node[anchor=south]{}(2.9,-0.9) ;
\draw[thick, ->] (2.2,-0.1).. controls (2.2,-0.1) and (2.2,-0.1) .. node[anchor=south]{}(3.9,-0.9) ;
\end{tikzpicture}
$$
as a sub-quiver. The underlying graph of the sub-quiver in the latter case is $\tilde{D_5}$ and it is an Euclidean
graph.
Since the underlying graph of both of them are not Dynkin diagrams, it follows that $H$ is of infinite corepresentation type.
\item[ii)]If $$D_1^{(2)}\cdot S(C_k)=S(C_k)+D_2^{(2)}$$ for some $D_2^{(2)}\in\mathcal{S}(2)$,  a similar argument shows that if $$D_2^{(2)}\cdot C_k=D_1^{(2)}+\k g_4+\k g_5,$$ where $g_4, g_5\in G(H)$, then $H$ is of infinite corepresentation. If not, we can consider the case that $$D_2^{(2)}\cdot C_k=D_1^{(2)}+D_3^{(2)},$$ where $D_3^{(2)}\in\mathcal{S}(2)$. Continue the process, we know that either $H$ is of infinite corepresentation type, or we can get a sub-quiver which contains infinite vertexes of $\k\mathrm{Q}(H)/J^2$ of the following form
$$
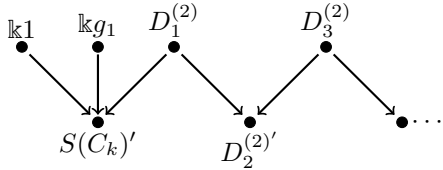
\begin{tikzpicture}
\filldraw [black] (0,0) circle (2pt) node[anchor=south]{$\k1$};
\filldraw [black] (1,0) circle (2pt)node[anchor=south]{$\k g_1$};
\filldraw [black] (2,0) circle (2pt)node[anchor=south]{$D_1^{(2)}$};
\filldraw [black] (1,-1) circle (2pt) node[anchor=north]{$S(C_k)^\prime$};
\filldraw [black] (3,-1) circle (2pt)node[anchor=north]{$D_2^{(2) ^\prime}$};
\filldraw [black] (4,0) circle (2pt)node[anchor=south]{$D_3^{(2)}$};
\filldraw [black] (5,-1) circle (2pt)node[anchor=west]{$\cdots$};
\draw[thick, ->] (0.1,-0.1).. controls (0.1,-0.1) and (0.1,-0.1) .. node[anchor=south]{}(0.9,-0.9) ;
\draw[thick, ->] (1,-0.1).. controls (1,-0.1) and (1,-0.1) .. node[anchor=south]{}(1,-0.9) ;
\draw[thick, ->] (1.9,-0.1).. controls (1.9,-0.1) and (1.9,-0.1) .. node[anchor=south]{}(1.1,-0.9) ;
\draw[thick, ->] (2.1,-0.1).. controls (2.1,-0.1) and (2.1,-0.1) .. node[anchor=south]{}(2.9,-0.9) ;
\draw[thick, ->] (3.9,-0.1).. controls (3.9,-0.1) and (3.9,-0.1) .. node[anchor=south]{}(3.1,-0.9) ;
\draw[thick, ->] (4.1,-0.1).. controls (4.1,-0.1) and (4.1,-0.1) .. node[anchor=south]{}(4.9,-0.9) ;
\end{tikzpicture}.
$$
For the latter case, it is in contradiction with the fact that $H$ is finite-dimensional.
\end{itemize}
\item[(III)]Finally, we focus on the case that $$S(C_k)\cdot C_k=\k1+D_1^{(3)}$$ in $\Bbb{Z}\mathcal{S}$ for some $D_1^{(3)}\in\mathcal{S}(3).$
\begin{itemize}
  \item[i)]If $$D_1^{(3)}\cdot S(C_k)=S(C_k)+D_1^{(2)}+D_2^{(2)},$$  where $D_1^{(2)}, D_2^{(2)}\in\mathcal{S}(2)$, then $$D_1^{(2)}\cdot C_k=D_1^{(3)}+\k g_1$$ and $$D_2^{(2)}\cdot C_k=D_1^{(3)}+\k g_2,$$where $g_1, g_2\in G(H).$ It follows from Lemma \ref{lemma:numberinquiver} (2) that the separated quiver for $\k\mathrm{Q}(H)/J^2$ either contains the Kronecker quiver as a sub-quiver, a sub-quiver whose underlying graph is $\tilde{A_{n}}$ for some $n\geq3$, or a sub-quiver of the following form
   $$
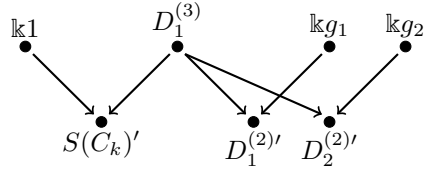
\begin{tikzpicture}
\filldraw [black] (-1,0) circle (2pt) node[anchor=south]{$\k1$};
\filldraw [black] (1,0) circle (2pt) node[anchor=south]{$D_1^{(3)}$};
\filldraw [black] (3,0) circle (2pt)node[anchor=south]{$\k g_1$};
\filldraw [black] (4,0) circle (2pt)node[anchor=south]{$\k g_2$};
\filldraw [black] (0,-1) circle (2pt) node[anchor=north]{$S(C_k)^\prime$};
\filldraw [black] (2,-1) circle (2pt)node[anchor=north]{$D_1^{(2) \prime}$};
\filldraw [black] (3,-1) circle (2pt)node[anchor=north]{$D_2^{(2) \prime}$};
\draw[thick, ->] (-0.9,-0.1).. controls (-0.9,-0.1) and (-0.9,-0.1) .. node[anchor=south]{}(-0.1,-0.9) ;
\draw[thick, ->] (0.9,-0.1).. controls (0.9,-0.1) and (0.9,-0.1) .. node[anchor=south]{}(0.1,-0.9) ;
\draw[thick, ->] (1.1,-0.1).. controls (1.1,-0.1) and (1.1,-0.1) .. node[anchor=south]{}(1.9,-0.9) ;
\draw[thick, ->] (1.1,-0.1).. controls (1.1,-0.1) and (1.1,-0.1) .. node[anchor=south]{}(2.9,-0.9) ;
\draw[thick, ->] (2.9,-0.1).. controls (2.9,-0.1) and (2.9,-0.1) .. node[anchor=south]{}(2.1,-0.9) ;
\draw[thick, ->] (3.9,-0.1).. controls (3.9,-0.1) and (3.9,-0.1) .. node[anchor=south]{}(3.1,-0.9) ;
\end{tikzpicture}.
$$
The underlying graph of the quiver in the latter case is $\tilde{E_6}$, which is an Euclidean graph. This means that $H$ is of infinite corepresentation type.
\item[ii)] If $$D_1^{(3)}\cdot S(C_k)=S(C_k)+D_2^{(3)}+\k g_1,$$ where $g_1\in G(H)$ and $D_2^{(3)}\in \mathcal{S}(3)$, we know that $\k g_1\cdot C_k$ contains $D^{(3)}_1$ with a nonzero coefficient in $\Bbb{Z}\mathcal{S}$. But $$\sqrt{\dim_{\k}(\k g_1)}\sqrt{\dim_{\k}(C_k)}<\sqrt{\dim_{\k}(D_1^{(3)})},$$ this leads to a contradiction. Therefore, this situation never happen.
\item[iii)]Suppose that $$D_1^{(3)}\cdot S(C_k)=S(C_k)+D_1^{(4)},$$ where $D_1^{(4)}\in\mathcal{S}(4)$, we can continue this process. Since $H$ is finite-dimensional, an argument similar to the one used in $(2)(II)$ shows that there exists some $n\geq3$ such that $$D_{1}^{(i)}\cdot S^{\alpha_i}(C_k)=D_{1}^{(i-1)}+D_{1}^{(i+1)},$$ holds for all $3\leq i\leq n$, and $$D_{1}^{(n+1)}\cdot S^{\alpha_{n+1}}(C_k)=D_{1}^{(n)}+E+F,$$ where $E, F \in\mathcal{S}, D_{1}^{(2)}=S(C_k), D_{1}^{(i)}\in\mathcal{S}(i)$ for $3\leq i\leq n+1$ and $\alpha_i=0$ when $i$ is even, $\alpha_i=1$ when $i$ is odd.\\
    When $n=2m$ for some $m\geq2$, a similar argument shows that $$\sqrt{\dim_{\k}(E)}= m+1, \;\;\sqrt{\dim_{\k}(F)}= m+1.$$ Notice that $E \cdot C_k$ contains at least one subcoalgebra $G$ with a nonzero coefficient besides $D_{1}^{(2m+1)}$, where $\sqrt{\dim_{\k}(G)}= 1$. Then we know that $G\cdot S(C_k)$ contains $E$, which is in contradiction with $\sqrt{\dim_{\k}(E)}\geq3$.\\
    When $n=2m+1$ for some $m\geq1$, since $E\cdot S(C_k)$ and $F\cdot S(C_k)$ contain $D_{1}^{(2m+2)}$ with a nonzero coefficient, it follows that $$\sqrt{\dim_{\k}(E)}\geq m+1, \;\;\sqrt{\dim_{\k}(F)}\geq m+1.$$  Without loss of generality, we can assume$$\sqrt{\dim_{\k}(E)}=m+2, \;\;\sqrt{\dim_{\k}(F)}=m+1.$$
    Note that $E\cdot S(C_k)$ contains at least one subcoalgebra $G$ with a nonzero coefficient besides $D_{1}^{(2m+2)}$, where $\sqrt{\dim_{\k}(G)}\leq 2$. Then we know that $G\cdot S(C_k)$ contains $E$, which means that $m=1$ or $m=2$ and $\sqrt{\dim_{\k}(G)}=2$.\\
    Based on the consideration above, we need only to consider the situations of $n=3$ and $n=5$. \\
    When $n=3$, $$D_1^{(4)}\cdot C_k=D_1^{(3)}+E+F,$$ where $E, F\in\mathcal{S}$. Since $E\cdot S(C_k)$ and $F\cdot S(C_k)$ contains $D_1^{(4)}$ with a nonzero coefficient, it follows that $$\sqrt{\dim_{\k}(E)}\geq2, \;\;\sqrt{\dim_{\k}(F)}\geq2.$$  Without loss of generality, suppose that $$\sqrt{\dim_{\k}(E)}=3,\;\;\sqrt{\dim_{\k}(F)}=2.$$  Then we have $$E\cdot S(C_k)= D_1^{(4)}+D_2^{(2)},$$
and
$$D_2^{(2)}\cdot S(C_k)= E+\k g_1,$$
where $g_1\in G(H)$, $D_2^{(2)}\in\mathcal{S}(2)$.
According to Lemma \ref{lemma:numberinquiver} (2), the separated quiver of $\k\mathrm{Q}(H)/J^2$ either contains the Kronecker quiver as a sub-quiver, a sub-quiver whose underlying graph is $\tilde{A_{n}}$ for some $n\geq3$ or a sub-quiver of the following type
$$
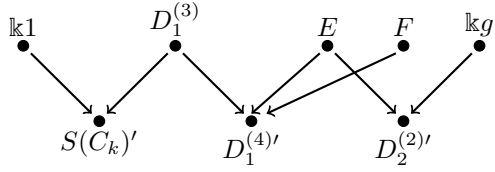
\begin{tikzpicture}
\filldraw [black] (-1,0) circle (2pt) node[anchor=south]{$\k1$};
\filldraw [black] (1,0) circle (2pt) node[anchor=south]{$D_1^{(3)}$};
\filldraw [black] (3,0) circle (2pt)node[anchor=south]{$E$};
\filldraw [black] (4,0) circle (2pt)node[anchor=south]{$F$};
\filldraw [black] (5,0) circle (2pt)node[anchor=south]{$\k g$};
\filldraw [black] (0,-1) circle (2pt) node[anchor=north]{$S(C_k)^\prime$};
\filldraw [black] (2,-1) circle (2pt)node[anchor=north]{$D_1^{(4) \prime}$};
\filldraw [black] (4,-1) circle (2pt)node[anchor=north]{$D_2^{(2) \prime}$};
\draw[thick, ->] (-0.9,-0.1).. controls (-0.9,-0.1) and (-0.9,-0.1) .. node[anchor=south]{}(-0.1,-0.9) ;
\draw[thick, ->] (0.9,-0.1).. controls (0.9,-0.1) and (0.9,-0.1) .. node[anchor=south]{}(0.1,-0.9) ;
\draw[thick, ->] (1.1,-0.1).. controls (1.1,-0.1) and (1.1,-0.1) .. node[anchor=south]{}(1.9,-0.9) ;
\draw[thick, ->] (2.9,-0.1).. controls (2.9,-0.1) and (2.9,-0.1) .. node[anchor=south]{}(2,-0.9) ;
\draw[thick, ->] (3.9,-0.1).. controls (3.9,-0.1) and (3.9,-0.1) .. node[anchor=south]{}(2.2,-0.9) ;
\draw[thick, ->] (3.1,-0.1).. controls (3.1,-0.1) and (3.1,-0.1) .. node[anchor=south]{}(3.9,-0.9) ;
\draw[thick, ->] (4.9,-0.1).. controls (4.9,-0.1) and (4.9,-0.1) .. node[anchor=south]{}(4.1,-0.9) ;
\end{tikzpicture}.
$$
The underlying graph of the sub-quiver in the latter case is $\tilde{E_7}$ and it is an Euclidean graph, which means that $H$ is of infinite corepresentation type.\\
When $n=5$, we have
$$D_1^{(4)}\cdot C_k= D_1^{(3)}+D_1^{(5)},$$
$$D_1^{(5)}\cdot S(C_k)= D_1^{(4)}+D_1^{(6)}$$ and
$$D_1^{(6)}\cdot C_k=D_1^{(5)}+E+F.$$ Without loss of generality, we can assume $$\sqrt{\dim_{\k}(E)}=4,\;\;\sqrt{\dim_{\k}(F)}=3.$$ It follows that $$E\cdot S(C_k)=D_1^{(6)}+D_3^{(2)}.$$
This means that the separated quiver for $\k\mathrm{Q}(H)/J^2$ either contains a sub-quiver whose underlying graph is $\tilde{A_{n}}$ for some $n\geq3$ or a sub-quiver of the following type
$$
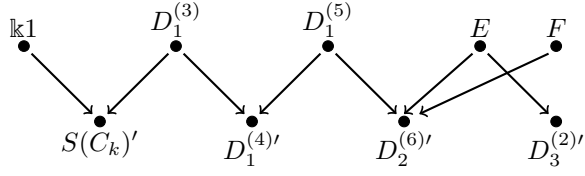
\begin{tikzpicture}
\filldraw [black] (-1,0) circle (2pt) node[anchor=south]{$\k1$};
\filldraw [black] (1,0) circle (2pt) node[anchor=south]{$D_1^{(3)}$};
\filldraw [black] (3,0) circle (2pt) node[anchor=south]{$D_1^{(5)}$};
\filldraw [black] (5,0) circle (2pt)node[anchor=south]{$E$};
\filldraw [black] (6,0) circle (2pt)node[anchor=south]{$F$};
\filldraw [black] (0,-1) circle (2pt) node[anchor=north]{$S(C_k)^\prime$};
\filldraw [black] (2,-1) circle (2pt)node[anchor=north]{$D_1^{(4) \prime}$};
\filldraw [black] (4,-1) circle (2pt)node[anchor=north]{$D_2^{(6) \prime}$};
\filldraw [black] (6,-1) circle (2pt)node[anchor=north]{$D_3^{(2) \prime}$};
\draw[thick, ->] (-0.9,-0.1).. controls (-0.9,-0.1) and (-0.9,-0.1) .. node[anchor=south]{}(-0.1,-0.9) ;
\draw[thick, ->] (0.9,-0.1).. controls (0.9,-0.1) and (0.9,-0.1) .. node[anchor=south]{}(0.1,-0.9) ;
\draw[thick, ->] (1.1,-0.1).. controls (1.1,-0.1) and (1.1,-0.1) .. node[anchor=south]{}(1.9,-0.9) ;
\draw[thick, ->] (2.9,-0.1).. controls (2.9,-0.1) and (2.9,-0.1) .. node[anchor=south]{}(2.1,-0.9) ;
\draw[thick, ->] (3.1,-0.1).. controls (3.1,-0.1) and (3.1,-0.1) .. node[anchor=south]{}(3.9,-0.9) ;
\draw[thick, ->] (4.9,-0.1).. controls (4.9,-0.1) and (4.9,-0.1) .. node[anchor=south]{}(4,-0.9) ;
\draw[thick, ->] (5.1,-0.1).. controls (5.1,-0.1) and (5.1,-0.1) .. node[anchor=south]{}(5.9,-0.9) ;
\draw[thick, ->] (5.9,-0.1).. controls (5.9,-0.1) and (5.9,-0.1) .. node[anchor=south]{}(4.2,-0.9) ;
\end{tikzpicture}.
$$
The underlying graph of the sub-quiver in the latter case is $\tilde{E_8}$ and it is still an Euclidean graph, which means that $H$ is of infinite corepresentation type.
  \end{itemize}
\end{itemize}
In conclusion, $H$ is of infinite corepresentation type.
\end{proof}
Recall that a basic cycle of length $n$ is a quiver with $n$ vertices $e_0, e_1, \cdots, e_{n-1}$ and $n$ arrows $a_0, a_1, \cdots a_{n-1}$, where the arrow $a_i$ goes from the vertex $e_i$ to the vertex $e_{i+1}$.
With the help of the proceeding three propositions and Corollary \ref{lemma:Pc=cP=1}, we can now obtain the following theorem.
\begin{theorem}\label{coro:equ}
Let $H$ be a finite-dimensional non-cosemisimple Hopf algebra over an algebraically closed field $\k$ with the dual Chevalley property and $\mathrm{Q}(H)$ be the link quiver of $H$. Then the following statements are equivalent:
\begin{itemize}
  \item[(1)]$H$ is of finite corepresentation type;
  \item[(2)]Every vertex in $\mathrm{Q}(H)$ is both the start vertex of only one arrow and the end vertex of only one arrow, that is, $\mathrm{Q}(H)$ is a disjoint union of basic cycles;
  \item[(3)]There is only one arrow $C\rightarrow \k1$ in $\mathrm{Q}(H)$ whose end vertex is $\k1$ and $\dim_{\k}(C)=1$;
  \item[(4)]There is only one arrow $\k1\rightarrow D$ in $\mathrm{Q}(H)$ whose start vertex is $\k1$ and $\dim_{\k}(D)=1$.
\end{itemize}
\end{theorem}
\begin{proof}
Combining Propositions \ref{prop:2arrow=infin}, \ref{Prop:finite corep} and \ref{dimC4,9}, we can prove the equivalence of (1) and (3). According to Lemma \ref{lemma:P^1=1^P} and Corollary \ref{lemma:Pc=cP=1}, we know the equivalence of (2), (3), and (4).
\end{proof}
Let $H_{(1)}$ be the link-indecomposable component containing $\k1$. Combining Proposition \ref{lemma:link1dim} and Theorem \ref{coro:equ}, we have:
\begin{corollary}\label{coro:H(1)pointedcycle}
A finite-dimensional non-cosemisimple Hopf algebra $H$ over $\k$ with the dual Chevalley property is of finite corepresentation type if and only if $H_{(1)}$ is a pointed Hopf algebra and the link quiver of $H_{(1)}$ is a basic cycle.
\end{corollary}
Recall that a finite-dimensional Hopf algebra $H$ over $\k$ is said to have the \textit{Chevalley property}, if radical $Rad(H)$ is a Hopf ideal. According to \cite[Propersition 4.2]{AEG01}, we know that $H$ has the Chevalley property if and only if $H^*$ has the dual Chevalley property.
\begin{theorem}\label{thm:nakayama}
A finite-dimensional Hopf algebra $H$ over an algebraically closed field $\k$ with the Chevalley property is of finite representation type if and only if $H$ is a Nakayama algebra.
\end{theorem}
\begin{proof}
The sufficiency follows immediately since it is known that every Nakayama algebra is of finite representation type. Next we show the necessity. In fact if $H$ has the Chevalley property, we know that $H^*$ has the dual Chevalley property. According to the proof of Proposition \ref{prop:2arrow=infin}, the Ext quiver of $H$ is the same as the link quiver of $H^*$. If $H$ is semisimple, the Ext quiver of $H$ contains no arrows. If $H$ is not semisimple, it follows from Theorem \ref{coro:equ} that the Ext quiver of $H$ is a finite union of basic cycles. Thus $H$ is a Nakayama algebra.
\end{proof}
Recall that a coalgebra $C$ is said to be coNakayama, if the dual algebra $C^*$ is a Nakayama algebra.
It is direct to see the following corollary.
\begin{corollary}\label{coro:conakayama}
A finite-dimensional Hopf algebra $H$ over an algebraically closed field $\k$ with the dual Chevalley property is of finite corepresentation type if and only if $H$ is coNakayama.
\end{corollary}
\subsection{Characterization}
Next we give a more accurate description for $H_{(1)}$ in the case that $H$ is a finite-dimensional non-cosemisimple Hopf algebra with the dual Chevalley property of finite corepresentation type.

In order to achieve our goals, we need the knowledge of monomial Hopf algebras and comonomial Hopf algebras (see \cite{CHYZ04} for details). An algebra $A$ is called \textit{monomial} if there exits a quiver $\mathrm{Q}$ and an admissible ideal $I$ generated by some paths such that $A\cong \k \mathrm{Q}/I$. A coalgebra $C$ is called \textit{comonomial} if $C^*$ is a monomial algebra. A finite-dimensional Hopf algebra is called a \textit{monomial} (resp. \textit{comonomial}) Hopf algebra if it is monomial (resp. comonomial) as an algebra (coalgebra).
\begin{remark}
Note that comonomial Hopf algebra in \cite{CHYZ04} was called monomial Hopf algebra. In order to not cause confusion, we recall the defintion of comonomial Hopf algebra.
\end{remark}
One of the key observation we need is the following lemma which was proved in \cite{CHYZ04}, which is true no matter when the characteristic of $\k$ is equal to $0$ or is equal to $p$.
\begin{lemma}\cite[Corollary 2.4]{CHYZ04}\label{lem:monomial=nakayama}
A non-semisimple Hopf algebra over $\k$ is a monomial Hopf algebra if and only if it is elementary and Nakayama.
\end{lemma}

It is known that a finite-dimensional Hopf algebra $H$ over an algebraically closed field is pointed if and only if $H^*$ is elementary. And according to the proof of Proposition \ref{prop:2arrow=infin}, the link quiver $\mathrm{Q}(H_{(1)})$ of $H_{(1)}$ agrees with the Ext quiver $\Gamma(H^*)^{\mathrm{a}}$ of $H^*$. Thus the following corollary is a direct consequence of Corollary \ref{coro:H(1)pointedcycle} and Lemma \ref{lem:monomial=nakayama}.
\begin{corollary}\label{coro:comonomial}
Let $H$ be a finite-dimensional non-cosemisimple Hopf algebra with the dual Chevalley property over an algebraically closed field $\k$. Then $H$ is of finite corepresentation type if and only if $H_{(1)}$ is a comonomial Hopf algebra.
\end{corollary}

The authors of \cite{CHYZ04} classify non-cosemisimple comonomial Hopf algebras via group data when the characteristic of $\k$ is zero. Let us briefly recall their results.

Let $\k$ be an algebraically closed field with characteristic 0, a \textit{group datum} (see \cite[Definition 5.3]{CHYZ04}) over $\k$ is defined to be a sequence $\alpha=(G, g, \chi, \mu)$ consisiting of
\begin{itemize}
  \item[(1)]a finite group $G$, with an element $g$ in its center;
  \item[(2)]a one-dimensional $\k$-representation $\chi$ of $G$; and
  \item[(3)]an element $\mu\in \k$ such that $\mu=0$ if $o(g)=o(\chi(g))$, and that if $\mu\neq 0$, then $\chi^{o(\chi(g))}=1$.
  \end{itemize}

For a group datum $\alpha=(G, g, \chi, \mu)$ over $\k$, the authors of \cite{CHYZ04} give the corresponding Hopf algebra structure $A(\alpha)$ as follow (for details, see \cite[5.7]{CHYZ04}).
Define $A(\alpha)$ to be an associative algebra with generators $x$ and all $h\in G$, with relations
$$
 x^d=\mu(1-g^d),\;\; xh=\chi(h)hx,\;\; \forall h\in G,
$$
with comultiplication $\Delta$, counit $\varepsilon$, and the antipode $S$ given by
$$
\Delta(g)=g\otimes g,\;\; \varepsilon(g)=1,\;\;
\Delta(x)=x\otimes 1+g\otimes x,\;\; \varepsilon(x)=0,\;\;
S(g)=g^{-1},\;\;S(x)=-g^{-1}x.
$$

Let $H$ be a non-semisimple comonomial Hopf algebra over $\k$, \cite[Lemma 5.2]{CHYZ04} permits us to introduce the following notion. A group datum $\alpha(H)=(G, g, \chi, \mu)$ is called an \textit{induced group datum} of $H$ (see \cite[Definition 5.5]{CHYZ04} for details) provide that
\begin{itemize}
  \item[(1)]$G=G(H)$, where $G(H)$ is the set of all the group-like elements of $H$;
  \item[(2)]there exists a non-trivial $(1, g)$-primitive element $x$ in $H$ such that
  $$
  x^d=\mu(1-g^d),\;\; xh=\chi(h)hx,\;\; \forall h\in G,
  $$
  where $d$ is the multiplicative order of $\chi(g)$.
  \end{itemize}
It is not difficult to verify that $H\cong A(\alpha(H))$.

As mentioned above, let us illustrate it with an example.
\begin{example}
Let $q\in \k$ be an $n$-th root of unit of order $d$. In \cite{AS98} and \cite{Rad77}, Radford and Andruskiewitsch-Schneider have considered the following Hopf algebra $A(n, d, \mu, q)$ which as an associative algebra is generated by $g$ and $x$ with relations
$$
g^n=1, \;\;\;\;x^d=\mu(1-g^d),\;\;\;\;xg=qgx.
$$
Its comultiplication $\Delta$, counit $\varepsilon$, and the antipode $S$ are given by
$$
\Delta(g)=g\otimes g,\;\; \varepsilon(g)=1,\;\;
\Delta(x)=x\otimes 1+g\otimes x,\;\; \varepsilon(x)=0,\;\;
S(g)=g^{-1},\;\;S(x)=-g^{-1}x.
$$
In fact, $(\Bbb{Z}_n, \overline{1}, \chi, \mu)$ with $\chi(\overline{1})=q$ is an induced group datum and  $A(n, d, \mu, q)=A(\Bbb{Z}_n, \overline{1}, \chi, \mu)$.
\end{example}
The following results gives a classification of non-semisimple comonomial Hopf algebra over an algebraically filed $\k$ of characteristic zero.
\begin{lemma}\emph{(}\cite[Theorem 5.9]{CHYZ04}\emph{)}\label{lemma:comonomial1-1}
Let $\k$ be an algebraically closed field with characteristic 0, there is a one-to-one corresponding between sets
$$\{\text{the isoclasses of non-cosemisimple comonomial Hopf algebras over }\k\}$$
and
$$\{\text{the isoclasses of group data over }\k\}.$$
\end{lemma}

Summarizing the above conclusions, we prove the following theorem of finite-dimensional Hopf algebras over an algebraically closed field $\k$ with characteristic 0 with the dual Chevalley property of finite corepresentation type, which is a generalization of \cite[Theorem 4.6]{LL07}.
\begin{theorem}\label{thm:finitecorep}
Let $\k$ be an algebraically closed field with characteristic zero. Then a finite-dimensional Hopf algebra $H$ over $\k$ with the dual Chevalley property is of finite corepresentation type if and only if either of the following conditions is satisfied:
\begin{itemize}
  \item[(1)]$H$ is cosemisimple;
  \item[(2)]$H$ is not cosemisimple and $H_{(1)}\cong A(n, d, \mu, q)$.
  \end{itemize}
\end{theorem}
\begin{proof}
Since $A(n, d, \mu, q)$ is a comonomial Hopf algebra, the if implication follows immediately from Corollary \ref{coro:comonomial}. It suffices to prove the only if part. According to Lemma \ref{lemma:comonomial1-1}, it is enough to find the induced datum of $H_{(1)}.$ Let $G(H_{(1)})$ be the set of group-like elements of $H_{(1)}$. If $H$ is a non-cosemisimple Hopf algebra of finite corepresntation type, it follows from Theorem \ref{coro:equ} and Corollary \ref{coro:H(1)pointedcycle} that $H_{(1)}$ is a pointed Hopf algebra and there exists a unique non-trivial $(1, g)$-primitive element $x$ for some $g\in G(H_{(1)})$. Without loss of generality, assume $$\mid G(H_{(1)})\mid=n.$$ Due to $H_{(1)}$ is link-indecomposable, the link quiver of $H_{(1)}$ is connected. This means that $G(H_{(1)})$ is a cyclic group whose generator is $g$. Thus the induced group datum of $H_{(1)}$ is $(\langle g\rangle, g, \chi, \mu)$ and we have
$$H_{(1)}\cong A(\langle g\rangle, g, \chi, \mu)\cong A(n, d, \mu, q).$$
\end{proof}
\begin{remark}
Let $\k$ be an algebraically closed field with characteristic zero.
\begin{itemize}
\item[(1)]Andruskiewitsch and Schneider conjectured that any finite-dimensional pointed Hopf algebra over $\k$ is generated in degree $1$ of its coradical filtration, i.e., by grouplike and skew-primitive elements \cite{AS02}. Suppose $H$ is a finite-dimensional Hopf algebra over $\k$ with the dual Chevalley property of finite corepresentation type. According to Theorem \ref{thm:finitecorep} and \cite[Corollary 4.10]{Li22a}, we can show that $H$ is generated in degree $1$ of its coradical filtration.
\item[(2)]Let $H$ be a finite-dimensional Hopf algebra over $\k$ with the dual Chevalley property. Recall that the rank of $H$ is defined to be $n$ if $\dim_{\k}(\k\otimes_{H_{0}} H_1)=n+1$ and $H$ is generated by $H_1$ as an algebra \cite{KR06}. It is not difficult to show that $H$ is of rank one if and only if $H$ is of finite corepresentation type.
\end{itemize}
\end{remark}
Finally, we focus on the above theorem in the case of that $\k$ is an algebraically closed field of characteristic $p$.
For any quiver $\mathrm{Q}$, we define $C_d{\mathrm{Q}}:=\bigoplus_{i=0}^{d-1}\k\mathrm{Q}(i)$ for $d\geq 2$, where $\mathrm{Q}(i)$ is the set of all paths of length $i$ in $\mathrm{Q}(i)$. It is not difficult to show that $C_d{\mathrm{Q}}$ is a subcoalgebra of path coalgabra $\k\mathrm{Q}$ (see \cite{CR02} for the definition of path coalgebra). We denote the basic cycle of length $n$ by $Z_n$ and denote $C_d(Z_n)$ by $C_d(n)$. By \cite[Theorem 1]{LY06}, we know that $C_d(n)$ admits a Hopf algebra structure if and only if there exists a primitive $d_0$-th root $q\in\k$ of unity with $d_0\mid n$ and a natural number $r\geq0$ such that $d=p^rd_0.$

Moreover, the authors of \cite{LY06} have given a description of the structures of comonomial Hopf algebras when the characteristic of $\k$ is not zero.
\begin{lemma}\emph{(}\cite[Theorem 4.2]{LY06}\emph{)}
Let $H$ be a non-cosemisimple comonomial Hopf algebra over an algebraically closed field $\k$ of character $p$. Then there exists a $d_0$-th primitive root $q\in\k$ of unit with $d_0\mid n$, $r\geq0$ and $d=p^rd_0$ such that
$$
H\cong C_d(n)\oplus \cdots \oplus C_d(n)
$$
as coalgebras and
$$
H\cong C_d(n)\# \k(G/N)
$$
as Hopf algebras, where $G=G(H)$, the set of group-like elements of $H$, and $N=G(C_d(n))$, the set of group-like elements of $C_d(n)$.
\end{lemma}
Based on the above lemma and the proof of Theorem \ref{thm:finitecorep}, we can now obtain the following theorem immediately.
\begin{theorem}\label{thm:finitecorepp}
Let $\k$ be an algebraically closed field of positive characteristic $p$. Then a finite-dimensional Hopf algebra $H$ over $\k$ with the dual Chevalley property is of finite corepresentation type if and only if either of the following conditions is satisfied:
\begin{itemize}
  \item[(1)]$H$ is cosemisimple;
  \item[(2)]$H$ is not cosemisimple and $H_{(1)}\cong C_d(n)$.
  \end{itemize}
\end{theorem}

\section{Examples}\label{section6}
In this section, we work over an algebraically closed field $\k$ of characteristic zero.
Let us first give a example which is of finite corepresentation type.
\begin{example}\label{example6.1}
Let $H$ be the Hopf algebra of dimension 16 appeared in \cite[Theorem 5.1]{CDMM04}. As an algebra, $H$ is generated by $c, b, x, y$ with relations:
$$c^2=1,\;\; b^2=1,\;\; x^2=\frac{1}{2}(1+c+b-cb),$$
$$cb=bc,\;\; xc=bx,\;\; xb=cx,$$
$$y^2=0,\;\; yc=-cy,\;\; yb=-by,\;\; yx=\sqrt{-1}cxy.$$
The coalgebra structure and antipode are given by:
$$\Delta(c)=c\otimes c,\;\; \varepsilon(c)=1,\;\; S(c)=c,$$
$$\Delta(b)=b\otimes b,\;\; \varepsilon(b)=1,\;\; S(b)=b,$$
$$\Delta(x)=\frac{1}{2}(x\otimes x+bx\otimes x+x\otimes cx-bx\otimes cx),\;\; \varepsilon(x)=1,\;\; S(x)=x,$$
$$\Delta(y)=c\otimes y+y\otimes 1,\;\; \varepsilon(y)=0,\;\; S(y)=-c^{-1}y.$$
Denote $E=\operatorname{span}\{x, bx, cx, bcx\}$, then $\mathcal{S}=\{\k1, \k c, \k b, \k bc,  E\}$. We give the corresponding multiplicative matrix $\E$ of $E$, where
$$
\E=\frac{1}{2}\left(\begin{array}{cc}
x+bx&x-bx\\
cx-bcx&cx+bcx
 \end{array}\right).
$$
We know that $\Bbb{Z}\mathcal{S}$ is a unital based ring and its multiplication table is shown below:
\begin{center}
\scalebox{1.0}{
\begin{tabular}{|c|c|c|c|c|c|c|c|c|c|c|c|c|c|c|}
\hline left action & $\k 1$ & $\k  c$ & $\k  b$ & $\k  bc$ & $E$\\
\hline$\k 1$ &$\k 1$ &$\k  c$&$k  b$ &$\k  bc$ &$E$\\
\hline$\k  c$ & $\k  c$& $\k  1$& $\k  bc$ &$\k  b$ & $E$\\
\hline$\k  b$ & $\k  b$& $\k  bc$&$\k  1$ &$\k c$ &$E$\\
\hline $\k  bc$ &$\k  bc$& $\k  b$   & $\k  c$  & $\k 1$ & $E$\\
\hline $E$&$E$&$E$&$E$&$E$&$\k 1+\k  c+\k  b+\k  bc$\\
\hline
\end{tabular}}\\
\title{TABLE 6.1. The multiplication table of $\Bbb{Z}\mathcal{S}$}
\end{center}
And in this example, $\mathcal{P}=\{(y), (cy), (by), (bcy), \X\}$, where
$$
\X=\frac{1}{2}\left(\begin{array}{cc}
xy+bxy&xy-bxy\\
bcxy-cxy&-cxy-bcxy
 \end{array}\right)
$$
is a non-trivial $(\E, \E)$-primitive matrix.
In this example, the link quiver of $H$ is shown below:
$$
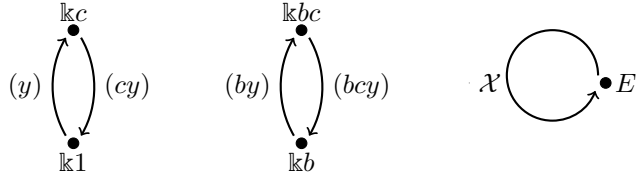
\begin{tikzpicture}
\filldraw [black] (6,0) circle (2pt) node[anchor=south]{$\k  c$};
\filldraw [black] (6,-1.5) circle (2pt)node[anchor=north]{$\k  1$};
\filldraw [black] (9,0) circle (2pt) node[anchor=south]{$\k bc$};
\filldraw [black] (9,-1.5) circle (2pt)node[anchor=north]{$\k  b$};
\filldraw [black] (13,-0.7) circle (2pt)node[anchor=west]{$E$};
\filldraw [black] (11.2,-0.7) circle (0pt)node[anchor=west]{$\X$};
\draw[thick, ->] (5.9,-1.4) .. controls (5.67,-1) and (5.67,-0.5) .. node[anchor=east]{$(y)$}(5.9,-0.1) ;
\draw[thick, ->] (6.1,-0.1) .. controls (6.33,-0.5) and(6.33,-1) .. node[anchor=west]{$(cy)$}(6.1,-1.4);
\draw[thick, ->] (8.9,-1.4) .. controls (8.67,-1) and (8.67,-0.5).. node[anchor=east]{$(by)$}(8.9,-0.1);
\draw[thick, ->] (9.1,-0.1) .. controls (9.33,-0.5) and (9.33,-1) ..  node[anchor=west]{$(bcy)$}(9.1,-1.4);
\draw[thick, ->] (12.9,-0.6) arc (0:340:0.6);
\end{tikzpicture}
$$
It follows from Theorem \ref{coro:equ} that $H$ is of finite corepresentation type. Moreover, due to $$ad_r(\frac{x+bx}{2})(c)=b \notin H_{(1)},$$ we know that
$H_{(1)}$ is not normal in $H$.
Thus this example gives a negative answer to \cite[Question 4.13]{Li22a}.
\end{example}
Next we introduce an example which is of infinite corepresentation type.
\begin{example}
Let $H$ be the Hopf algebra of dimensional 32 which is generated by $z, y, t, p_1, p_2$ satisfying the following relations:
$$z^{2}=1,\;\; y^{2}=1,\;\; t^{2}=1,\;\;  z y=y z,\;\;  t z=z t,\;\; t y=y t,$$
$$z p_{1}=p_{1} z,\;\; y p_{1}=p_{1} y,\;\; t p_{1}=-p_{1} t,\;\; z p_{2}=p_{2} z,\;\; y p_{2}=p_{2} y,\;\; t p_{2}=-p_{2} t,$$
$$p_{1}^{2}=\lambda\left(1-z\right),\;\; p_{2}^{2}=-\lambda\left(1-z\right),\;\;  p_{1} p_{2}+p_{2} p_{1}=0. $$
The coalgebra structure and antipode are given by:
$$\Delta(z)=z \otimes z,\;\; \Delta(y)=y \otimes y,\;\; \varepsilon(z)=\varepsilon(y)=1,$$
$$\Delta(t)=\frac{1}{2}\left[(1+y) t \otimes t+(1-y) t \otimes z t\right],\;\;  \varepsilon(t)=1,$$
$$S(z)=z,\;\;  S(y)=y,\;\;  S(t)=\frac{1}{2}\left[(1+y) t+(1-y) z t\right], $$
$$\Delta\left(p_{1}\right)=p_{1} \otimes 1+\frac{1}{2}\left(1+z\right) t \otimes p_{1}+\frac{1}{2}\left(1-z\right) y t \otimes p_{2},$$
$$\Delta\left(p_{2}\right)=p_{2} \otimes 1+\frac{1}{2}\left(1+z\right) y t \otimes p_{2}+\frac{1}{2}\left(1-z\right) t \otimes p_{1},$$
$$\varepsilon(p_1)=\varepsilon(p_2)=0,$$
$$S(p_1)=-\frac{1}{4}\left[(1+y)t+(1-y)zt\right]\left[(1+z)p_1+y(1-z)p_2\right],$$
$$S(p_2)=-\frac{1}{4}\left[(1+y)t+(1-y)zt\right]\left[y(1+z)p_2+(1-z)p_1\right].$$
Let $z=x^2$, it is straightforward to verify that $H$ is a Hopf subalgebra of $\mathfrak{U}_{17}(I_{17})$ in \cite[Definition 6.17]{ZGH21}. And in fact, $H$ is a Hopf subalgebra of Hopf algebra ${\mathfrak{U}_{17}(I_{17})}_{(1)}$, where ${\mathfrak{U}_{17}(I_{17})}_{(1)}$ denote the link-indecomposition component containing $\k1$.\\
Denote $E=\operatorname{span}\{t, zt, yt, zyt\}$, then $\mathcal{S}=\{\k 1, \k  z, \k  y, \k  zy,  E\}$. We give the corresponding multiplicative matrix $\E$ of $E$, where
$$
\E=\frac{1}{2}\left(\begin{array}{cc}
t+yt&t-yt\\
zt-zyt&zt+zyt
 \end{array}\right).
$$
Then $\Bbb{Z}\mathcal{S}$ is a unital based ring and its multiplication table is shown below:
\begin{center}
\scalebox{1.0}{
\begin{tabular}{|c|c|c|c|c|c|c|c|c|c|c|c|c|c|c|}
\hline left action & $\k 1$ & $\k  y$ & $\k  z$ & $\k  zy$ & $E$\\
\hline$\k 1$ &$\k 1$ &$\k  y$&$k  z$ &$\k  zy$ &$E$\\
\hline$\k  y$ & $\k  y$& $\k  1$& $\k zy$ &$\k  z$ & $E$\\
\hline$\k z$ & $\k  z$& $\k  zy$&$\k 1$ &$\k  y$ &$E$\\
\hline $\k  zy$ &$\k  zy$& $\k  z$   & $\k y$  & $\k 1$ & $E$\\
\hline $E$&$E$&$E$&$E$&$E$&$\k 1+\k  z+\k  y+\k  zy$\\
\hline
\end{tabular}}\\
\title{TABLE 6.2.  The multiplication table of $\Bbb{Z}\mathcal{S}$}
\end{center}
And in this example, $\mathcal{P}^\prime=\{\X_1, \X_2, \X_3, \X_4, \X_5, \X_6, \X_7, \X_8\}$, where
$$
\X_1=\frac{1}{2}\left(\begin{array}{c}
p_1+p_2\\
p_1-p_2
 \end{array}\right),
\X_2=\frac{1}{2}\left(\begin{array}{c}
p_1z-p_2z\\
p_1z+p_2z
 \end{array}\right),
$$
$$
 \X_3=\frac{1}{2}\left(\begin{array}{c}
p_1y+p_2y\\
p_2y-p_1y
 \end{array}\right),
 \X_4=\frac{1}{2}\left(\begin{array}{c}
p_2zy-p_1zy\\
p_1zy+p_2zy
 \end{array}\right),
$$
$$
\X_5=\frac{1}{4}\left(\begin{array}{cc}
(p_1+p_2)(1+y)t+(p_1-p_2)(1-y)zt&(p_1+p_2)(1-y)t+(p_1-p_2)(1+y)zt
 \end{array}\right),
$$
$$
\X_6=\frac{1}{4}\left(\begin{array}{cc}
(p_1+p_2)(1-y)zt+(p_1-p_2)(1+y)t&(p_1+p_2)(1+y)zt+(p_1-p_2)(1-y)t
 \end{array}\right),
$$
$$
\X_7=\frac{1}{8}\left(\begin{array}{cc}
(p_1-p_2)(1+y)t-(p_1+p_2)(1-y)z t&(p_1-p_2)(1-y)t-(p_1+p_2)(1+y)z t
 \end{array}\right),
 $$
 $$
\X_8=\frac{1}{8}\left(\begin{array}{cc}
(p_1-p_2)(1-y)z t+(p_1+p_2)(1+y)t&(p_1+p_2)(1+y)z t-(p_1+p_2)(1-y)t
 \end{array}\right).
 $$
The link quiver of $H$ is shown below:
$$
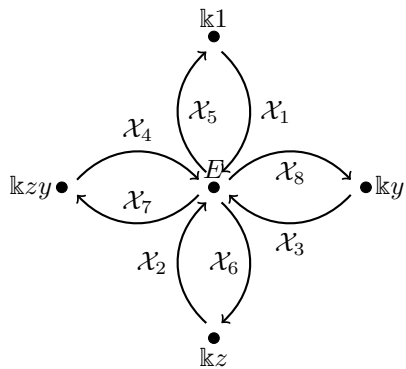
\begin{tikzpicture}
\filldraw [black] (0,0) circle (2pt) node[anchor=south]{$E$};
\filldraw [black] (0,2) circle (2pt)node[anchor=south]{$\k  1$};
\filldraw [black] (0,-2) circle (2pt) node[anchor=north]{$\k  z$};
\filldraw [black] (2,0) circle (2pt)node[anchor=west]{$\k y$};
\filldraw [black] (-2,0) circle (2pt)node[anchor=east]{$\k  zy$};
\draw[thick, ->] (-0.1,0.2) .. controls (-0.6,0.67) and (-0.6,1.34) .. node[anchor=west]{$\X_5$}(-0.1,1.8) ;
\draw[thick, ->] (0.1,1.8) .. controls (0.6,1.34) and (0.6,0.67) .. node[anchor=west]{$\X_1$}(0.1,0.2) ;
\draw[thick, ->] (-0.1,-1.8) .. controls  (-0.6,-1.34)and  (-0.6,-0.67)..  node[anchor=east]{$\X_2$}(-0.1,-0.2);
\draw[thick, ->] (0.1,-0.2) .. controls (0.6,-0.67) and (0.6,-1.34) .. node[anchor=east]{$\X_6$}(0.1,-1.8) ;
\draw[thick, ->] (0.2,0.1) .. controls (0.67,0.6) and (1.34,0.6) .. node[anchor=north]{$\X_8$}(1.8,0.1) ;
\draw[thick, ->] (1.8,-0.1) .. controls (1.34,-0.6) and (0.67,-0.6) ..node[anchor=north]{$\X_3$}(0.2,-0.1) ;
\draw[thick, ->] (-1.8,0.1) .. controls (-1.34,0.6) and (-0.67,0.6) ..  node[anchor=south]{$\X_4$}(-0.2,0.1);
\draw[thick, ->]  (-0.2,-0.1).. controls (-0.67,-0.6) and (-1.34,-0.6)  ..node[anchor=south]{$\X_7$}(-1.8,-0.1)  ;
\end{tikzpicture}
$$
According to Theorem \ref{coro:equ}, we know that $H$ is of infinite corepresentation type.
\end{example}

\section*{}
\subsection*{Author Contributions}
All authors contributed to all aspects of this project.
\subsection*{Funding}
The second author was supported by Zhejiang Provincial Natural Science Foundation of China [grant number LQ23A010003]. The third author was supported by NSFC [grant number 12271243].
\subsection*{Availability of data and materials}
Not applicable.
\section*{Declarations}
\subsection*{Competing interests}
The authors declare no competing interests.
\subsection*{Ethical Approval} Not applicable.


\begin{thebibliography}{99}
\setlength{\itemsep}{0em}
\bibitem{ABM12}Ardizzoni, A., Beattie, M., Menini, C.: Gauge deformations for Hopf algebras with the dual Chevalley property. J. Algebra Appl. \textbf{11}, 1-37 (2012)
\bibitem{AC20}Assem, I., Coelho, F.U.: Basic Representation Theory of Algebras. Springer, Cham (2020)
\bibitem{AEG01}Andruskiewitsch, N., Etingof, P., Gelaki, S.: Triangular Hopf algebras with the Chevalley property. Machigan Math. J. \textbf{49}, 277-298 (2001)
\bibitem{AGM17}Andruskiewitsch, N., Galindo, C., M{\"u}ller, M.: Examples of finite-dimensional Hopf algebras with the dual Chevalley property. Publ. Mat. \textbf{61}, 445-474 (2017)
\bibitem{ARS95}Auslander, M., Reiten, I., Smal{\o}, S.: Representation Theory of Artin Algebras. Cambridge University Press, Cambridge (1995)
\bibitem{AS98}Andruskiewitsch, N., Schneider, H.-J.: Lifting of quantum linear spaces and pointed Hopf algebras of order $p^3$. J. Algebra \textbf{209}, 658-691 (1998)
\bibitem{AS02}Andruskiewitsch, N., Schneider, H.-J.: Pointed Hopf algebras. In: New directions in Hopf algebras. Cambridge University Press, Cambridge, 1-68 (2002)
\bibitem{Ari05}Ariki, S.: Hecke algebras of classical type and their representation type. Proc. London Math. Soc. \textbf{91}, 355-413 (2005)
\bibitem{Ari17}Ariki, S.: Representation type for block algebras of Hecke algebras of classical type. Adv. Math. \textbf{317},  823-845 (2017)
\bibitem{Ari21}Ariki, S.: Tame block algebras of Hecke algebras of classical type. J. Aust. Math. Soc. \textbf{111}, 179-201 (2021)
\bibitem{ASS06}Assem, I., Simson, D., Skowro{\'{n}}ski, A.: Elements of the Representation Theory of Associative Algebra Vol.1. Cambridge univercity Press, Cambridge (2006)
\bibitem{Ben98}Benson, D.: Representation and Cohomology I. Cambridge University Press, Cambridge (1998)
\bibitem{BD82}Bondarenko, V., Drozd, Y.: Representation type of finite groups. J. Sov. Math. \textbf{20}, 2515-2528 (1982)
\bibitem{CDMM04}Calinescu, C., Dascalescu, S., Masuoka, A., Menini, C.: Quantum lines over non-cocommutative
cosemisimple Hopf algebras. J. Algebra \textbf{273}, 753-779 (2004)
\bibitem{CHYZ04}Chen, X., Huang, H., Ye, Y., Zhang, P.: Monomial Hopf algebras. J. Algebra \textbf{275}, 212-232 (2004)
\bibitem{CHZ06}Chen, X., Huang, H., Zhang, P.: Dual Gabriel theorem with applications. Sci. China Ser. A \textbf{49},  9-26 (2006)
\bibitem{CM97}Chin, W., Montgomery, S.: Basic coalgebras. In: Modular interfaces, Riverside, CA (1995);
AMS/IP Stud. Adv. Math., 4, Amer. Math. Soc., pp. 41-47. Providence, RI (1997)
\bibitem{Cil93}Cibils, C.: A quiver quantum group. Comm. Math. Phys. \textbf{157}, 459-477 (1993)
\bibitem{Cil97}Cibils, C.: Half-quantum groups at roots of unit, path algebras, and representation type. Int. Math. Res. Not. IMRN \textbf{12}, 541-553 (1997)
\bibitem{CR97}Cibils, C., Rosso, M.: Algebres des chemins quantique. Adv. Math. \textbf{125}, 171-199 (1997)
\bibitem{CR02}Cibils, C., Rosso, M.: Hopf quivers. J. Algebra \textbf{254}, 241-251 (2002)
\bibitem{DEMN99}Doty, S. R., Erdmann, K., Martin, S., Nakano, D. K.: Representation type of Schur algebras. Math. Z. \textbf{232}, 137-182 (1999)
\bibitem{Dro86}Drozd, Y.: Tame and wild matrix problems. In: Representations and Quadratic Forms. Inst. Math., Acad. Sciences. pp. 39-74, Ukrainian SSR, Kiev (1979); Am. Math. Soc. Transl. \textbf{128}, 31-55 (1986)
\bibitem{EGNO15}Etingof, P., Gelaki, S., Nikshych, D., Ostrik, V.: Tensor Categories. Amer. Math. Soc., Providence (2015)
\bibitem{EN01}Erdmann, K., Nakano, D. K.: Representation type of $q$-Schur algebras. Trans. Amer. Math. Soc. \textbf{353}, 4729-4756 (2001)
\bibitem{Erd90}Erdmann, K.: Blocks of tame representation type and related algebras. In: Lecture Notes in Math, vol. 1428, Springer-Verlag Berlin (1990)
\bibitem{Et02}Etingof, P.: On Vafa's theorem for tensor categories. Math. Res. Lett. \textbf{9}, 651-658 (2001)
\bibitem{Far06}Farnsteiner, R.: Polyhedral groups, Mckey quivers and the finite algebraic groups with tame princial blocks. Invent. Math. \textbf{166}, 27-94 (2006)
\bibitem{FS02}Farnsteiner, R., Skowronski, A.: Classification of restricted Lie algebras with tame principal block. J. Reine. Angew. Math. \textbf{546}, 1-45 (2002)
\bibitem{FS07}Farnsteiner, R., Skowronski, A.: Galois actions and blocks of tame infinitesimal group schemes. Trans. Amer. Math. Soc. \textbf{359}, 5867-5898 (2007)
\bibitem{FV00}Farnsteiner, R., Voigt, D.: On cocommutative Hopf algebras of finite representation type. Adv. Math. \textbf{155}, 1-22 (2000)
\bibitem{FV03}Farnsteiner, R., Voigt, D.: On infinitesimal groups of tame representation type. Math. Z. \textbf{244}, 479-513 (2003)
\bibitem{GS98}Green, E., Solberg, {\O}.: Basic Hopf algebras and quantum groups. Math. Z. \textbf{229}, 45-76 (1998)
\bibitem{Hig54}Higman, D.: Indecomposable representation at characteristic $p$. Duke Math. J. \textbf{21}, 377-381 (1954)
\bibitem{HL09}Huang, H., Liu, G.: On the structure of tame graded basic Hopf algebras II. J. Algebra \textbf{321}, 2650-2669 (2009)
\bibitem{KOS11}Kashuba, I., Ovsienko, S., Shestakov, I.: Representation type of Jordan algebras. Adv. Math. \textbf{226}, 385-418 (2011)
\bibitem{KR06}Krop, K., Radford, D. E.: Finite-dimensional Hopf algebras of rank one in characteristic zero. J. Algebra \textbf{302}, 214-230 (2006)
\bibitem{Lar71}Larson, R. G.: Characters of Hopf algebras. J. Algebra \textbf{17}, 352-368 (1971)
\bibitem{Li22a}Li, K.: The link-indecomposable components of Hopf algebras and their products.
J. Algebra \textbf{593}, 235-273 (2022)
\bibitem{Li22b}Li, K.: Note on invariance and finiteness for the exponent of Hopf algebras. Comm. Algebra \textbf{50}, 484-497 (2022)
\bibitem{Liu06}Liu, G.: On the structure of tame graded basic Hopf algebras. J. Algerba \textbf{299}, 841-853 (2006)
\bibitem{Liu13}Liu, G.: Basic Hopf algebras of tame type. Algebr. Represent. Theory \textbf{16}, 771-791 (2013)
\bibitem{LL07}Li, F., Liu, G.: Pointed Hopf algebras of finite corepresentation type and their classifications. Proc. Amer. Math. Soc. \textbf{135}, 649-657 (2007)
\bibitem{LL22}Li, K., Liu, G.: On the antipode of Hopf algebras with the dual Chevalley property. J. Pure Appl. Algebra \textbf{226}, 1-15 (2022)
\bibitem{LZ19}Li, K., Zhu, S.: On the exponent of finite-dimensional non-cosemisimple Hopf algebras. Comm.
Algebra \textbf{47}, 4476-4495 (2019)
\bibitem{LY06}Liu, G., Ye, Y.: Monomial Hopf algebras over fields of positive characteristic. Sci. China Ser. A \textbf{49}, 320-329 (2006)
\bibitem{Man88}Manin, Y. I.: Quantum Groups and Non-Commutative Geometry. Montral, QC, Canada: Centre de Recherches mathmatiques, Universit de Montral (1988)
\bibitem{Mom13}Mombelli, M.: Families of finite-dimensional Hopf algebras with the Chevalley property. Algebr. Represent. Theory \textbf{16}, 421-435 (2013)
\bibitem{Mon93}Montgomery, S.: Hopf Algebras and Their Actions on Rings. Published for the Conference Board of the Mathematical Sciences, Washington, DC; by the American Mathematical Society, Providence, RI (1993)
\bibitem{Mon95}Montgomery, S.: Indecomposable coalgebras, simple comodules and pointed Hopf algebras. Proc. Amer. Math. Soc. \textbf{123}, 2343-2351 (1995)
\bibitem{Ost03}Ostrik, V.: Module categories, weak Hopf algebras and modular invariants. Transform. Groups \textbf{8}, 177-206 (2003)
\bibitem{OZ04}Oystaeyen, F. V., Zhang, P.: Quiver Hopf algebras. J. Algebra \textbf{280}, 577-589 (2004)
\bibitem{Rad75}Radford, D. E.: On the coradical of a finite-dimensional Hopf algebra. Proc. Amer. Math. Soc. \textbf{53}, 9-15 (1975)
\bibitem{Rad77}Radford, D. E.: Operators on Hopf algebras. Amer. J. Math. \textbf{99}, 139-158 (1977)
\bibitem{Rad78}Radford, D. E.: On the structure of commutative pointed Hopf algebra. J. Algebra \textbf{50}, 284-296 (1978)
\bibitem{Rin75}Ringel, C. M.: The representation type of local algebras. In: Representation of Algebras, Lecture Notes in Math., vol. 488, pp. 282-305, Springer (1975)
\bibitem{Rin78}Ringel, C. M.: Finite dimensional hereditary algebras of wild representation type. Math. Z. \textbf{161}, 235-255 (1978)
\bibitem{Sch17}Schott, J. R.: Matrix Analysis for Statistics. John Wiley \& Sons, Inc., Hoboken, NJ (2017)
\bibitem{Sut94}Suter, R.: Modules over $U_{q}(\mathfrak{sl}_2)$. Commun. Math. Phys. \textbf{163}, 359-393 (1994)
\bibitem{Swe69}Sweedler, M. E.: Hopf Algebras. W. A. Benjamin, Inc., New York (1969)
\bibitem{WW18}Wang, L., Wang, X.: Indicators of Hopf algebras in positive characteristic. Arch. Math. (Basel) \textbf{111}, 485-491 (2018)
\bibitem{Xia97}Xiao, J.: Finite-dimensional representations of $U_{t}(\mathfrak{sl}_2)$ at roots of unity. Can. J. Math. \textbf{49}, 772-787 (1997)
\bibitem{ZGH21}Zheng, Y., Gao, Y., Hu, N.: Finite-dimensional Hopf algebras over the Hopf
algebra $H_{b:1}$ of Kashina. J. Algebra \textbf{567}, 613-659 (2021)
\end{thebibliography}
\end{document}